\documentclass[11pt]{article}
\usepackage[utf8]{inputenc}
\usepackage{amsmath}
\usepackage{graphicx}
\usepackage[margin = 1 in]{geometry}
\usepackage{amsfonts}
\usepackage{amsthm}
\usepackage{mathrsfs}
\usepackage{bigints}
\usepackage{amssymb}
\usepackage{booktabs}
\usepackage{xspace}
\usepackage{tikz}
\usepackage{tikz-cd}
\usepackage{enumerate}
\usepackage{hyperref}
\usepackage{upgreek}
\usepackage{verbatim}
\usepackage{color}
\newcommand\numberthis{\addtocounter{equation}{1}\tag{\theequation}}
\newcommand{\R}{\mathbb{R}}
\newcommand{\RR}{\mathbb{R}}
\newcommand{\PP}{\mathbb{P}}

\newcommand{\QQ}{\mathbb{Q}}
\newcommand{\EE}{\mathbb{E}}

\newcommand{\N}{\mathbb{N}}
\newcommand{\NN}{\mathbb{N}}

\newcommand{\cla}{\mathcal{A}}
\newcommand{\cld}{\mathcal{D}}
\newcommand{\clg}{\mathcal{G}}

\newcommand{\clr}{\mathcal{R}}
\newcommand{\clp}{\mathcal{P}}
\newcommand{\clb}{\mathcal{B}}
\newcommand{\cln}{\mathcal{N}}
\newcommand{\clx}{\mathcal{X}}
\newcommand{\clf}{\mathcal{F}}
\newcommand{\Om}{\Omega}
\newcommand{\om}{\omega}
\newcommand{\bs}[1]{\boldsymbol{#1}}

\allowdisplaybreaks

\theoremstyle{plain}
\newtheorem{theorem}{Theorem}[section]
\newtheorem{lemma}[theorem]{Lemma}
\newtheorem{corollary}[theorem]{Corollary}

\newtheorem{proposition}[theorem]{Proposition}

\theoremstyle{definition}
\newtheorem{remark}{Remark}

\theoremstyle{definition}

\newtheorem{definition}{Definition}

\title{Rank Based Routing in Large Server Systems under Extreme Congestion}
\begin{document}
\author{Sayan Banerjee, Amarjit Budhiraja, and Eva Loeser}
\date{}
\maketitle

\begin{abstract}
We study a large system of $n$ parallel queues operating in an extreme heavy traffic regime, where each server processes jobs at rate $n$ and jobs arrive at a central dispatcher at rate $n^2 - (a-b)\sqrt{n}$, for fixed $a>b>0$. Thus the heavy traffic slack, per queue, is $(a-b)/\sqrt{n}$ which is much smaller than in the standard heavy traffic regime where the slack is $O(\sqrt{n})$. Upon arrival, jobs are routed according to a marginal join-the-shortest-queue policy: a small stream of arrivals, at rate $b\sqrt{n}$, are routed to the current shortest queue, while the remaining stream of jobs, at rate $n^2 - a\sqrt{n}$ , are routed independently at random. This policy substantially reduces communication overhead relative to full join-the-shortest-queue routing and provides a mechanism for providing high quality of service to premium jobs by routing them to the shortest queue, while simultaneously performing load balancing in the overall system.

Under the diffusive scaling, we prove  limit theorems for the ranked queue-length process and the associated gap process as $n \to \infty$. The limiting object is an infinite-dimensional reflected Atlas process, with reflection at the origin and rank-based drift that acts on the lowest particle. A key feature of the limit is that its dynamics depend only on the parameter $b$ governing the  stream of shortest-queue jobs, while the parameter $a$ only appears through the selection of invariant distributions. We establish well-posedness of the reflected infinite Atlas model and characterize a one-parameter family of product-form stationary distributions, parametrized by  $a$ and $b$, for its gap process.

To connect the limiting diffusion with the stationary behavior of the prelimit queueing system, we introduce a closely related {\em system with pauses} that agrees with the original dynamics at the diffusion scale and, at the same time, admits an exact representation as an open Jackson network. This representation yields explicit finite-$n$ stationary gap distributions, whose heavy traffic limits select the corresponding product-form invariant laws of the infinite reflected Atlas process. As consequences, we obtain sharp asymptotics for the lowest ranked queues, the system imbalance, and the average queue length. These formulas illustrate the tradeoff between reduced communication cost and load-balancing performance by comparing marginal shortest-queue routing with random routing and full join-the-shortest-queue policies, whose steady-state behavior in this heavy traffic regime is not yet well understood.

\noindent\newline


\end{abstract}

\section{Introduction}
\label{sec:intro}
\subsection{Model}
We consider a large system of parallel queues that is critically loaded and in which jobs arrive to a central dispatcher who then routes these jobs to one of the queues using some load balancing algorithm. 
In recent years, the rapid expansion of large scale service infrastructures, such as cloud computing platforms, distributed data centers, and data storage and retrieval systems \cite{MukherjeeBorstVanLeeuwaardenWhiting2020,JoshiKumari2016,GuptaHarcholBalterSigmanWhitt2007,banbudest}, has led to growing interest in the development and analysis of load balancing schemes for such parallel server systems.  One of the most well studied load balancing algorithms is the join-the-shortest-queue (JSQ) scheme, under which the dispatcher routes each arrival to the current shortest queue in the system. JSQ enjoys many optimality properties (see \cite{MukherjeeBorstVanLeeuwaardenWhiting2020} and references therein). However, for large systems with high throughput, implementing a JSQ policy leads to significant communication costs. To address this, the recent work \cite{banbudest} introduced Marginal JSQ (MJSQ) routing policies in which a vanishingly small fraction of arrivals per unit time are routed using the JSQ policy while the remaining are routed at random (RR). Since RR does not require any knowledge of the state of the system, the communication cost for implementing the MJSQ is significantly lower than that for JSQ.  
One of the central results in \cite{banbudest} shows that, although MJSQ collects the state information far less frequently than JSQ, it nevertheless attains load sharing performance that is asymptotically close to the optimal performance of JSQ.

The paper \cite{banbudest} considered a setting with a fixed number of servers and with the system operating under a normal heavy traffic regime.
In order to describe the precise setting, we introduce a scaling parameter $n \in \NN$ which will govern both the system size and the heavy traffic intensity.
The number of queues is denoted by $k(n)$  and the total arrival rate to the system is $k(n)(n - c(n))$ 
for some functions $k: \NN \to \NN$ and $c: \NN \to \RR_+$, and each of the server processes jobs at rate $n$.
With this notation, \cite{banbudest} considered the setting where $k(n)=K$ for some fixed $K \in \NN$ and $c(n) = v\sqrt{n}/K$ for some $v \in (0,\infty)$.
(The paper \cite{banbudest} also studied the overloaded case where $v<0$, however we will not consider that case here.) Under the standard Markovian assumptions (exponential interarrival and service time distributions that are mutually independent), \cite{banbudest} studied the MJSQ routing policy in which,  $v=\alpha K-\beta$, $\alpha, \beta>0$, and arrivals occur to each queue (independently of others) at rate $n - \alpha\sqrt{n}$, with the modification that the shortest queue gets an additional arrival stream of rate $\beta \sqrt{n}$.
Among other results, the paper \cite{banbudest} shows that the steady state queue-length (averaged over the $K$ queues)  under the above MJSQ policy is approximately 
$$\sqrt{n} \left( \frac{K-1}{K\alpha} + \frac{1}{v}\right).$$
With the same net arrival rate, the corresponding quantity for the policy that routes all jobs at random (RR) is approximately $\sqrt{n}K/v$ and for the (full) JSQ policy it is approximately $\sqrt{n}/v$. Thus, implementing MJSQ leads to an order $K$ improvement over RR, while its performance is close to that achieved by JSQ for large $K$ and $\alpha$. 

We note that, in the above finite server setting with $k(n)=K$, the arrival rate per queue is 
$(n-v\sqrt{n}/K)$ and therefore queues experience a standard heavy traffic intensity of the form $1 -\kappa/\sqrt{n}$ for some fixed slack parameter $\kappa>0$ not depending on $n$.
In the current work, we are interested in a heavy traffic regime that is more extreme than the standard heavy traffic setting considered above. Specifically, we consider the case where the heavy traffic slack parameter $\kappa= v^*/n$, for some $v^*>0$.  Thus, in comparison to \cite{banbudest}, the system experiences a much higher level of congestion. In this regime, for fixed system size, the steady state queue lengths for the shortest queue grow at the order of $n^{3/2}$ under any of the routing policies discussed above. However, one expects that as the system size grows, a MJSQ policy even with a very infrequent exploration of the full state of the system, should lead to significantly smaller queue lengths  for the shortest queue (or more generally for the $d$-shortest queues for fixed $d$). E.g., when $k(n)=n$, one would anticipate the shortest queue to be of order $\sqrt{n}$ while the longest queue to be of the order $n^{3/2}\log n$ (see Corollary \ref{corstat}).    This separation in queue length scales suggests that, even under extreme congestion, the occasional state space exploration inherent in MJSQ  provides a mechanism for providing high quality of service to premium jobs by routing them to to these persistently shorter queues, while maintaining a low communication cost overhead and simultaneously performing load balancing in the overall system.

In order to explore and analyze such  settings of large systems under high congestion we consider the case where $c(n)= v^*/\sqrt{n}$ and $k(n)=n$.   We will consider a MJSQ policy where all except $O(\sqrt{n})$ jobs of the $O(n^2)$ jobs arriving per unit time are routed randomly and the remaining jobs are routed to the shortest queue. Specifically, 
writing $v^* = a-b$ where $a>b$,
in the $n$-th system,  the arrival rate to the shortest queue (where ties in queue length are broken in lexicographical order) will be
$$n - \frac{a}{\sqrt{n}} + b \sqrt{n}.$$
and the arrival rate to the remaining queues will be
$n- \frac{a}{\sqrt{n}}$. Once there, they are served according to the first come first serve (FCFS) scheme.
As before, we assume that the interarrival and service times are exponentially distributed and are mutually independent. We remark that although the shortest queue experiences an  arrival rate which is substantially higher than the rate of service, the overall system is stable since the identity of the shortest queue among the $n$ queues changes dynamically over time.

Let, for $i \in [n]:= \{1, \ldots, n\}$, $X_i^n(t)$ be the length of the $i$-th queue at time $t$ according to some initial labeling of servers  and let $X_{(i)}^n(t)$ be the length of the $i$-th shortest queue at time $t$ (breaking ties with lexicographical ordering).
 Specifically, for $t \ge 0$, let $\pi_t$ be the (random measurable) permutation on $[n]$ such that, for $i,j \in [n]$,
$\pi_t(i) <\pi_t(j)$ if and only if either 
$X^n_{i}(t) < X^n_{j}(t)$ or $X^n_{i}(t) = X^n_{j}(t)$ and $i<j$. Then $X_{(i)}^n(t) = X^n_{\pi_t(i)}(t)$.

For $i=1,2,\ldots n,$ let $Z^n_i(t) := X_{(i)}^n(t)-X_{(i-1)}^{n}(t)$, with $X^n_{(0)}(t):=0$.

\subsection{Summary of main results}
Our first objective is to study the asymptotic transient behavior of the above processes under a diffusion scaling.
For this, we define 
\begin{equation}\label{eq:diffscal}
\hat{X}^n_{i}(t) = \frac{X_{i}^n(t)}{\sqrt{n}}, \hspace{4mm}
\hat{X}^n_{(i)}(t) = \frac{X_{(i)}^n(t)}{\sqrt{n}} \hspace{4mm}\text{and}\hspace{4mm} \hat{Z}_i^n(t) = \frac{Z_i^n(t)}{\sqrt{n}}, \;\; n \in \NN,
\end{equation}
and let $\hat{\boldsymbol{X}}^n_{\uparrow}(t) = (\hat{X}^n_{(1)}(t), \hat{X}^n_{(2)}(t), \ldots , \hat{X}^n_{(n)}(t))$, and
$\hat{\boldsymbol{Z}}^n(t) = (Z_{1}^n(t), Z_{2}^n(t), \ldots, Z_{n}^n(t))$.
By regarding a vector in $\RR_+^n$ as a vector in $\RR_+^{\infty}$ with zeroes in all entries after the $n$-th coordinate, we view $\hat{\boldsymbol{X}}^n_{\uparrow}$ and $\hat{\boldsymbol{Z}}^n$ as  stochastic processes with sample paths in
$D([0,\infty), \RR_+^{\infty})$ where $\RR_+^{\infty}$ is equipped with the usual product topology.
Our first result characterizes the asymptotic behavior of
$\hat{\boldsymbol{X}}^n_{\uparrow}$ and
$\hat{\boldsymbol{Z}}^n$ in $D([0,\infty), \RR_+^{\infty})$ in terms of certain infinite dimensional reflected diffusions ${\boldsymbol{X}}_{\uparrow}$ and ${\boldsymbol{Z}}$. A precise definition of these limit processes will be provided in Section \ref{setupsect}, however, roughly speaking, the process ${\boldsymbol{X}}_{\uparrow}$ can be viewed as an infinite Atlas model with the lowest particle reflected at zero, while the process ${\boldsymbol{Z}}$ gives the sequence of gaps associated with this process.
Both the finite and infinite Atlas models, without reflection at zero, arise naturally as canonical diffusion limits of rank-based particle systems. They have been studied extensively (see, e.g., \cite{fernholz2009stochastic, pitmanpal, Sarantsev2017InfiniteSystems}) and they admit explicit, product-form stationary distributions \cite{pitmanpal,SarTsai}, with well-studied domains of attraction \cite{DemboJaraOlla2019,BanerjeeBudhiraja2022}.
We establish the wellposedness of reflected analogues of these infinite dimensional diffusion processes that arise in our analysis (see Theorem \ref{uniquenessthm}) and prove the convergence of processes $\hat{\boldsymbol{X}}^n_{\uparrow}$ and
$\hat{\boldsymbol{Z}}^n$ in distribution, in $D([0,\infty), \RR_+^{\infty})$, to ${\boldsymbol{X}}_{\uparrow}$ and
${\boldsymbol{Z}}$ (with suitable initial conditions), respectively (see Theorem \ref{convergencethm}).
We remark here that, with our choice of topology on $\RR^{\infty},$ the above result gives the convergence 
$$(\hat{X}^n_{(1)}(\cdot), \hat{X}^n_{(2)}(\cdot), \ldots , \hat{X}^n_{(d)}(\cdot)) \Rightarrow (\hat{X}_{(1)}(\cdot), \hat{X}_{(2)}(\cdot), \ldots , \hat{X}_{(d)}(\cdot))$$
in $D([0,\infty), \RR_+^{d})$ for any $d \in \NN$ (and similarly for $\hat{\boldsymbol{Z}}^n$).
However, from this one cannot deduce the convergence of quantities such as the average queue-lengths: $\frac{1}{n}\sum_{i=1}^n \hat{X}^n_{(i)}(t)$ which in fact will not typically be tight under the scaling considered above.

We note that the processes $\hat{\boldsymbol{X}}^n_{\uparrow}$ and
$\hat{\boldsymbol{Z}}^n$ are characterized in terms of two key parameters: $(a, b)$. However, as is seen from the description of the limit processes ${\boldsymbol{X}}_{\uparrow}$ and ${\boldsymbol{Z}}$ in Section \ref{setupsect}, the probability laws of these  latter processes, although depending on $b$, are invariant under the choice of the parameter $a$.  This is not surprising, as the system is much closer to criticality than in a standard heavy traffic regime. Indeed, consider the basic example of a M/M/1 queue under heavy traffic with arrival rate of $n-\alpha(n)$ and service rate of $n$. Then, if $\alpha(n) = v\sqrt{n}$ for some $v \in \RR$, the queue-length process under the diffusion scaling converges to a reflected Brownian motion with drift $v$, but if $\alpha(n) = o(\sqrt{n})$, the limit diffusion is a reflected Brownian motion with $0$ drift and thus does not depend on the particular choice of $\alpha(n)$.  However, in striking contrast with the latter M/M/1 queue setting, where the limiting diffusion does not admit a stationary distribution, the limiting reflected Atlas model for the many-server system considered here admits a rich family of stationary distributions.

One of the reasons to consider diffusion approximations for queuing systems is to gain information about the steady state behavior of the system under a diffusion scaling.  Indeed, one of the main results of \cite{banbudest} shows that the reflected diffusion  studied in that work gives not only a good transient approximation for the ranked queue-length processes but also a good approximation over long time horizons. Specifically, it is proven that the steady state distribution of the ranked queue-length processes, under a diffusion scaling, converge to that of the approximating diffusion.  In order to establish similar connections for the large scale parallel server systems considered in the current work, it is natural to first study the long time properties of the limit process ${\boldsymbol{X}}_{\uparrow}$, or equivalently, that of ${\boldsymbol{Z}}$.  In Theorem \ref{stationarydistreflectedatlasresult} we show that, unlike the finite dimensional diffusion that arises in the analysis of parallel service systems in \cite{banbudest} which has a unique stationary distribution, the infinite dimensional diffusion describing ${\boldsymbol{Z}}$ has {\em many} stationary distributions. In fact, this theorem provides a one parameter family of stationary distributions $\{\mu_{a}, a \in (b,\infty)\}$ given as products of exponential laws.  The use of the notation $a$ to denote the parameter is intentional: the stationary distribution $\mu_a$ is related to the diffusion scaled process $\hat{\boldsymbol{Z}}^n$ that is associated with the parameters $(a,b)$ in a precise sense which will be described below.  Roughly speaking, the dependence on $a$, which, as noted above, vanishes in the transient asymptotic analysis of the process $\hat{\boldsymbol{Z}}^n$, reveals itself when one considers the long time behavior of this process under a diffusion scaling through a particular selection of a stationary distribution of the limit diffusion ${\boldsymbol{Z}}$. We remark that the stationary distribution $\mu_a$ also depends on the parameter $b$ (which corresponds to the fact that the distribution of ${\boldsymbol{Z}}$ depends on $b$) and to make this dependence explicit we will occasionally write $\mu_a$ as $\mu_{a,b}$.

We now give a precise description of the connection between the diffusion scaled pre-limit processes associated with parameters $(a,b)$ and the stationary distribution $\mu_{a,b}$ of the limit diffusion $\hat{\boldsymbol{Z}}$.  One of the challenges in the analysis of the dynamics of the ranked queue-length processes $\hat{\boldsymbol{X}}^n_{\uparrow}$ is to understand the evolution of the system when there are ties. With the ordering permutation map $\pi_t$ as before, suppose that immediately before a jump at time instant $t$, the state of the ranked queues takes the form
\begin{equation}\label{eq:tie}
X_{\pi_{t-}(i-1)}^{n}(t-) <X_{\pi_{t-}(i)}^{n}(t-)=X_{\pi_{t-}(i+1)}^{n}(t-)= \cdots =X_{\pi_{t-}(i+j)}^{n}(t-)< X_{\pi_{t-}(i+j+1)}^{n}(t-).
\end{equation}
Namely, there is a tie of length $j+1$ for the ranks $i, i+1, \ldots i+j$.
Under the original system considered in our work, an arrival at instant $t$ to one of the queues $\pi_{t-}(i), \ldots , \pi_{t-}(i+j)$ is credited to queue $\pi_{t}(i+j)$ whereas a departure from one of the queues $\pi_{t-}(i), \ldots , \pi_{t-}(i+j)$ is debited from queue $\pi_{t}(i)$. We remark that, in the heavy traffic regime that we consider, the Lebesgue measure of the amount of time any two queues are tied converges to $0$ in the limit (cf. \cite[Lemma 4.2]{banbudest}). In view of this negligibility, we consider a slightly different model that is easier to analyze in terms of the long time behavior. In this model, which we refer to as the {\em system with pauses}, when the system state prior to a jump at time instant $t$ takes the form in \eqref{eq:tie}, all arrivals to queues $\pi_{t-}(i), \ldots , \pi_{t-}(i+j-1)$ are paused and all servers for queues $\pi_{t-}(i+1), \ldots , \pi_{t-}(i+j)$ pause processing of jobs; rest of the system evolution remains the same. We note that, due to the slowdown of activity because of the pauses,  ties break more slowly in the modified system than in the original system.  As noted above, the time instants when the queues are tied is in a certain sense negligible. In fact, we show in Theorem \ref{convergencethm2} that with this modified dynamics, the processes corresponding to $\hat{\boldsymbol{X}}^n_{\uparrow}$ and
$\hat{\boldsymbol{Z}}^n,$ denoted $\hat{\boldsymbol{Y}}^n$ and $\hat{\boldsymbol{U}}^n$, respectively, have the same weak limits in $D([0,\infty), \RR_+^{\infty})$ as for the original dynamics.  The main reason for considering the modified dynamics is that we are able to show that there is a direct correspondence between the process $\hat{\boldsymbol{U}}^n$ (under the modified dynamics) and a certain Jackson network viewed under a diffusive scaling (see Section \ref{convergenceofstationariessect}). Using classical results for Jackson networks \cite{Kelly1979}, we then provide an explicit formula for the stationary distribution for $\hat{\boldsymbol{U}}^n$ and furthermore show that this stationary distribution converges to the invariant measure $\mu_{a,b}$ of the infinite dimensional diffusion ${\boldsymbol{Z}}$
(see Corollary \ref{jackcor}).
This result provides a precise connection between the long time behavior of the queuing system with parameters $(a,b)$ under a diffusive scaling and the stationary distribution $\mu_{a,b}$ of the limit diffusion ${\boldsymbol{Z}}$. As a second such connection, we show in Corollary \ref{cor: mix} that
the process
$\hat{\boldsymbol{Z}}^n$, under the original dynamics, when initiated with the stationary distribution for the modified dynamics of  $\hat{\boldsymbol{U}}^n$, has an {\em approximate stationarity} property, namely, its law remains approximately unchanged over compact time intervals, and the process converges to the stationary gap process of the reflected infinite Atlas model with invariant law  $\mu_{a,b}$. This result, together with the fact that the limit diffusion has multiple stationary distributions, also suggests a certain  metastability phenomenon which we hope to explore in future work; see Section \ref{sec:future}. In Corollary \ref{corstat}, we present asymptotics for the lowest ranked queues, system imbalance (discrepancy between the maximum and minimum queuelengths) and average queuelength. Using these, we compare the performance of MJSQ with random routing (RR) and join-the-shortest-queue (JSQ) in Remark \ref{rem:qual}. When $a-b \ll a$, we find that MJSQ does significantly better in load-balancing and reducing the average queuelength in comparison to RR, while keeping the shortest queue of a similar length despite a much larger  stream of arrivals routed to it. JSQ is expected to perform substantially better in terms of average queue length, since it explores the state space much more aggressively, albeit at the cost of higher communication overhead. However, in the extreme heavy-traffic regime considered here, the steady-state behavior of JSQ is not well understood (see \cite{raj2024exponential} for progress in this direction), and therefore a precise rigorous comparison is currently unavailable; see however Remark \ref{rem:qual}.

From \cite{banbudest} it follows that the Markov process $\hat{\boldsymbol{Z}}^n$ also has a unique stationary distribution.
One may ask whether, as for $\hat{\boldsymbol{U}}^n$, the stationary distribution of $\hat{\boldsymbol{Z}}^n$ also converges 
to $\mu_{a,b}$. Proving such a result presents substantial challenges; see next section and Remark \ref{rem:compare}, however in Corollary \ref{cor:mjsq}, we prove an analogous result for systems with a fixed number of servers. Namely, we show that the heavy traffic limits of the stationary gap distributions for the original system and the system with pauses, when the number of servers is fixed, coincide, thereby providing a justification for  using the latter system as a proxy for the behavior of the original one on the infinite time horizon.

\subsection{Future directions}
\label{sec:future}
In future work, we aim to establish a result of the form in Corollary \ref{cor:mjsq} when $k(n)=n$, over the infinite time horizon. 
Namely, we would like to show that the stationary distributions of the system with pauses and the original system asymptotically coincide in the case where $k(n)=n$.
This involves substantial work as it requires a refined control on mixing times and rates of convergence to equilibrium for the two high dimensional Markov processes, and estimating the time spent by each system at ties; see Remark \ref{rem:compare}. We will also expect to see a \emph{metastability} behavior, under which the system with parameters $a,b$, started sufficiently close (in an appropriate sense) to $\mu_{a',b}$ for some $a' \neq a$, should exhibit an approximate stationarity property over long time intervals (growing with $n$) before approaching  its unique invariant distribution that approximates $\mu_{a,b}$. This type of behavior is plausible since $\mu_{a',b}$ is also invariant for the limit reflected Atlas model, which is agnostic to the choice of $a$.

We also plan to study interpolations between the JSQ and MJSQ policies by tuning up the routing probability of incoming jobs to the shortest queue under the MJSQ scheme. This will provide a flexible framework for balancing performance and communication costs, while also offering a mathematically tractable approach to heavy-traffic regimes for JSQ that are currently not fully well understood (see Remark \ref{rem:qual}).


\subsection{Related work}

Load balancing policies for parallel-server systems have attracted much interest in recent years.
The much-studied JSQ policy is optimal in many respects, including vanishing mean waiting times in both the fluid and diffusion regimes \cite{MukherjeeBorstVanLeeuwaardenWhiting2020}.
However, it comes at a significant communication cost because it requires the router to query each server in the system every time it routes an incoming job.
To address this, Power-of-Choice, or JSQ(d), algorithms have received significant interest.
In the JSQ(d) load balancing scheme, when a job arrives to the system, the router queries $d \in \N$ queues chosen at random. Then, it routes the job to the shortest of those $d$ queues.
It is known that, for a parallel server system in heavy traffic with a fixed number of servers $K$, diffusion limit of JSQ(d) for any $d \in \{2,3,...,K\}$ is the same. Namely, JSQ(2) performs as well as JSQ in the fixed-server diffusion limit. 
In contrast, the many-server diffusion limit is highly sensitive to the value of the choice parameter $d$. Specifically, consider a sequence of $n$-server systems operating under the JSQ$(d(n))$ load-balancing policy. In order for the diffusion limit to coincide with that of the full JSQ policy, it is necessary that $\frac{d(n)}{\sqrt{n}\log n}\to\infty$ as $n\to\infty$; see \cite{MukherjeeBorstVanLeeuwaardenWhiting2020}. If $d(n)$ grows more slowly, then one observes a broad range of distinct diffusive asymptotic regimes \cite{bhamidi2022near}, in sharp contrast with the behavior under full JSQ.

Other load balancing policies that have smaller communication costs and similar optimality properties to JSQ include Join the Idle Queue, in which, instead of querying some number of queues and sending a job to the shortest queue, the router remembers which queues are idle, and sends jobs to idle queues first \cite{MukherjeeBorstvanLeeuwaardenWhiting2016}.
Also recall that, in the fixed-number-of-servers setting, the MJSQ policy studied in this paper was introduced in \cite{banbudest}, where it was shown to enjoy optimality properties similar to those of JSQ, while incurring substantially lower communication costs.

Here we considered a Markovian setting where interarrival and service times are exponentially distributed, but the harder setting of general  distributions has also been studied in the literature \cite{bramson2012asymptotic, atar2025rank}.

A closely related many-server randomized load balancing system, with a Poisson stream of arrivals but general service time distribution, was recently studied in \cite{atar2025invariance}.  In their model, a small fraction of arrivals is
routed using a Power-of-Choice rule, while the remaining jobs are routed uniformly at random.
Under a critical scaling, the authors obtain a hydrodynamic limit for the empirical distribution of
diffusion-scaled queuelengths and invariance principles leading to reflected McKean--Vlasov
diffusions.

\subsection{Notation}
The following notation will be used. For $n \in \NN$, we denote
$[n]:= \{1, 2, \ldots, n\}$. For a Polish space $S$, $D([0,\infty), S)$ will denote the space of right continuous functions with left limits, from $[0,\infty)$ to $S$, equipped with the usual Skorohod topology. Similarly, we denote by $C([0,\infty), S)$, the space of continuous functions from $[0,\infty)$ to $S$, equipped with the local uniform topology.

\section{Setting and Main Results}
\label{setupsect}
We consider a sequence of parallel server systems indexed by $n$. In the $n$-th system, there are $n$ parallel servers, the arrival rate to the shortest queue (where ties in queue length are broken in lexicographical order) is
$n - \frac{a}{\sqrt{n}} + b \sqrt{n},$
and the arrival rate to the remaining queues will be
$n- \frac{a}{\sqrt{n}}$. Once there, they are served according to the first come first serve (FCFS) scheme with each server processing jobs at rate $n$.
We assume that the interarrival and service times are exponentially distributed and are mutually independent.
In order to describe the state evolution of the system, we consider a collection $(\cla_i, \cld_i,  i \in [n])$ of mutually independent Poisson random measures (PRM) on $\RR_+\times \RR_+$ with intensity  as the Lebesgue measure $dt\times dr$ on $\RR_+\times \RR_+$.
We can give a representation for the queue-length process $\{X^n_i, i \in [n]\}$ in terms of these PRM as follows.  Suppose that the initial queue-length vector is given as $\{x^n_i, i \in [n]\}$, i.e. $X^n_i(0) = x^n_i$, $i \in [n]$. We assume without loss of generality that the initial vector is ordered, namely $x^n_1 \le \cdots \le x^n_n$. Then
\begin{equation}
\label{eq:xnit}
X^n_i(t) = x^n_i + A^n_i(t)  - D^n_i(t), \; i \in [n], t\ge 0,
\end{equation}
where
\begin{align*}
A^n_i(t) &= \int_{[0,t]\times \RR_+} 1_{[0, n - an^{-1/2} + bn^{1/2}1_{\{X^n_i(s-) = X^n_{(1)}(s-)\}}]}(r) \cla_i(ds\, dr), \; i \in [n]\\
D^n_i(t) &= \int_{[0,t]\times \RR_+} 1_{[0, n]}(r) 1_{\{X^n_i(s-) \neq 0\}} \cld_i(ds\, dr), \; i \in [n]\\
\end{align*}
where  $X_{(i)}^n(t)$ is the length of the $i$-th shortest queue at time $t$ (breaking ties with lexicographical ordering).
As before, for $i\in [n]$, $Z^n_i:= X_{(i)}^n-X_{(i-1)}^{n}$, with $X^n_{(0)}:=0$.
Diffusion scaled processes $\hat X^n_i, \hat X^n_{(i)}, \hat Z^n_i$ are defined by \eqref{eq:diffscal}. 

We consider stochastic processes $\hat{\boldsymbol{X}}^n_{\uparrow}$ and
$\hat{\boldsymbol{Z}}^n$  with sample paths in $D([0,\infty), \RR_+^{\infty})$
as
$$\hat{\boldsymbol{X}}^n_{\uparrow}(t) = (\hat X^n_{(1)}(t), \ldots , \hat X^n_{(n)}(t), 0, \ldots),$$
$$\hat{\boldsymbol{Z}}^n(t) = (\hat Z^n_1(t), \ldots , \hat Z^n_n(t), 0, \ldots).$$
Our first result will show that the processes $\hat{\boldsymbol{X}}^n_{\uparrow}$ and $\hat{\boldsymbol{Z}}^n$ converge, in distribution, in  $D([0,\infty), \RR_+^{\infty})$, to certain infinite dimensional reflected diffusions.
We now present these candidate limit processes.

We begin by introducing the {\em reflected infinite Atlas} process that will be used to describe the limit processes. Consider a process with sample paths in $C([0,\infty), \RR_+^{\infty})$, denoted as $\boldsymbol{X}= (X_1, X_2, \ldots)$, and described by stochastic differential equations of the form
\begin{equation}\label{eq:unranksde}
            dX_i = 1(X_i=X_{(1)})\delta dt +\sqrt{2} dB_i + dL_i
        \end{equation}
        where $\delta \in \RR$, $\{B_i(\cdot): i \in \N\}$ are a family of independent Brownian motions, $L_i$ is the local time process for $X_i$ at $0$, and $X_{(1)}(t) = \inf_{i \in \NN} X_i(t)$. A precise
        meaning to the solution of \eqref{eq:unranksde} is given in a weak sense as follows. By a filtered probability space 
        $\Xi = (\Om, \clf, \PP, \{\clf_t\}_{t\ge 0})$ we mean a probability space $(\Om, \clf, \PP)$ equipped with a filtration
        $ \{\clf_t\}_{t\ge 0}$ satisfying the usual conditions. By a collection $\{B_i, i \in \NN\}$ of Brownian motions on this space we mean that $\{B_i, i \in \NN\}$ are mutually independent one dimensional standard Brownian motions and each $B_i$ is a $\clf_t$-martingale.
        A vector $\boldsymbol{x}= (x_1, x_2, \ldots) \in \RR_+^{\infty}$ is said to be {\em rankable} if there exists a bijection $\pi:\NN \to \NN$ such that $x_{\pi(1)} \le x_{\pi(2)} \le \cdots$.

\begin{definition}
\label{def: reflected infinite atlas model}
    Fix $\mu \in \clp(\RR_+^{\infty})$. By a weak solution of \eqref{eq:unranksde} with initial distribution $\mu$, we mean a $\clf_t$-adapted stochastic process $\boldsymbol{X}= (X_1, X_2, \ldots)$ with sample paths in $C([0,\infty), \RR_+^{\infty})$, given on some filtered probability space $\Xi = (\Om, \clf, \PP, \{\clf_t\}_{t\ge 0})$ on which we are given a collection $\{B_i, i \in \NN\}$ of Brownian motions such that
    \begin{enumerate}
        \item The initial vector $\boldsymbol{X}(0):= (X_1(0),X_2(0),...)$ has probability  distribution $\mu$.
        \item for every $t\geq 0,$ the vector $\boldsymbol{X}(t):=(X_1(t),X_2(t),...)$ is rankable almost surely, with a (jointly) measurable bijection map $\pi_t$.
        \item  the process $X_i$ satisfies the SDE
       \eqref{eq:unranksde}, namely with $X_{(1)}(t) := X_{\pi_t(1)}(t)$,
       $$X_i(t) =  X_i(0) + \delta\int_0^t 1_{\{X_i(s)=X_{(1)}(s)\}} ds +\sqrt{2}B_i(t) + L_i(t), \; t \ge 0, \; i \in \NN,$$
       where $L_i$ is  a continuous nondecreasing $\clf_t$-adapted process satisfying
       $L_i(0) =0$ and $$\int_{\RR_+} 1_{\{X_i(t)>0\}} dL_i(t)=0.$$
        \end{enumerate}
        We say that weak uniqueness holds for \eqref{eq:unranksde} with initial distribution $\mu$ if whenever $\boldsymbol{X}$ and $\tilde{\boldsymbol{X}}$ are two weak solutions of \eqref{eq:unranksde} with initial distribution $\mu$, then the probability laws of 
        $\boldsymbol{X}$ and $\tilde{\boldsymbol{X}}$ on $C([0,\infty), \RR_+^{\infty})$ are the same.
\end{definition}

Our first result shows the wellposedness (weak existence and uniqueness of solutions) of \eqref{eq:unranksde} for suitable initial conditions.
Following \cite{pitmanpal, Sarantsev2017InfiniteSystems} we consider initial configurations $\boldsymbol{x} = (x_1, x_2, \ldots) \in \RR_+^{\infty}$ that satisfy $x_1 \le x_2 \le \cdots$ and
$\sum_{i=1}^{\infty} \exp(-\alpha x_i^2) <\infty$ for all $\alpha>0$. 
We will denote the space of all such configurations by $\Lambda_0$. We denote by $\clp(\Lambda_0)$ as the space of all probability measures on $\Lambda_0$.  Theorem \ref{uniquenessthm} below gives  wellposedness of \eqref{eq:unranksde} for $\mu \in \clp(\Lambda_0)$.

\begin{theorem}
\label{uniquenessthm}
Let $\mu \in \clp(\Lambda_0)$.
Then there exists a weak solution to \eqref{eq:unranksde} with initial distribution $\mu$ and weak uniqueness holds with this initial distribution.
\end{theorem}
Proof of the theorem is given in Section \ref{sec:thm1pf}.
\begin{remark}\label{rem:weakn}
We will need to also consider a finite dimensional analogue of \eqref{eq:unranksde}, namely for $N\in \NN$, we will consider a $\RR_+^N$ valued continuous process $\bs{X}^{(N)} = (X_1^{(N)}, \ldots , X_N^{(N)})$ which is a weak solution of \eqref{eq:unranksde} (with $i \in [N]$) with some initial distribution $\mu^{(N)} \in \clp(\RR_+^N)$. The weak solution (and its uniqueness) are defined exactly as in Definition \ref{def: reflected infinite atlas model} with $\RR_+^{\infty}$ and infinite vectors replaced by $\RR_+^N$ and $N$-dimensional vectors (for this setting part 2 of Definition \ref{def: reflected infinite atlas model} holds always). The proof of existence and uniqueness of weak solutions in this finite dimensional case (in this case no requirement is needed on the initial distribution) is proved exactly as Theorem \ref{uniquenessthm} and we omit the details.
\end{remark}

Let $\boldsymbol{X}$ be a weak solution to \eqref{eq:unranksde}, with some initial distribution $\mu \in \clp(\Lambda_0)$ given on some filtered probability space with Brownian motions $\{B_i\}$.
Since, by definition of a weak solution, $\boldsymbol{X}(t)=(X_1(t),X_2(t),...)$ is rankable almost surely, we can define
$$X_{(i)}(t) := X_{\pi_t(i)}(t), \; i \in \NN, t \ge 0.
$$
Let
$$Z_i(t) := X_{(i)}(t) - X_{(i-1)}(t), \; i \in \NN, \; t \ge 0$$
where we set $X_{(0)}(t)=0$ for all $t\ge 0$.
We denote
\begin{equation}\label{eq:hatxhatz}
{\boldsymbol{X}}_{\uparrow} := (X_{(1)}, X_{(2)}, \ldots), \;
{\boldsymbol{Z}} := (Z_1, Z_2, \ldots).
\end{equation}
We will refer to $\bs{X}$, $\bs{X}_{\uparrow}$, $\bs{Z}$ as the reflected Atlas system, the ranked reflected Atlas system, and the gap process for the reflected Atlas system, respectively.

Our next result characterizes the limits of $\hat{\boldsymbol{X}}^n_{\uparrow}$ and
$\hat{\boldsymbol{Z}}^n$ in terms of ${\boldsymbol{X}}_{\uparrow}, {\boldsymbol{Z}}$.
\begin{theorem}
\label{convergencethm}
Suppose that the initial queue-length vector $\{X^n_i(0), i \in [n]\}$ satisfies
$X^n_1(0) \le X^n_2(0) \le \cdots \le X^n_n(0)$ and the distribution of
$\hat{\boldsymbol{X}}^n(0) = (X^n_1(0)/\sqrt{n}, \ldots,  X^n_n(0)/\sqrt{n}, 0, \ldots)$ converges in $\clp(\RR_+^{\infty})$ to
some $\mu \in \clp(\Lambda_0)$.
Furthermore, suppose that
\begin{equation}\label{condnn}
\limsup_{N\to \infty} \limsup_{n\to \infty} \mathbb{P}\left(\sum_{i=N}^{n} e^{-c\hat X^n_i(0)} > \delta\right) = 0 \mbox{ for all } c >0, \delta>0.
\end{equation}
Let $\boldsymbol{X}$  be the weak solution of  \eqref{eq:unranksde} with initial distribution $\mu$ and with $\delta=b$. Define ${\boldsymbol{X}}_{\uparrow}$ and ${\boldsymbol{Z}}$ through
\eqref{eq:hatxhatz}.
Then, as $n\to \infty$,
$(\hat{\boldsymbol{X}}^n_{\uparrow}, \hat{\boldsymbol{Z}}^n)$ converges in distribution, in $D([0,\infty), \RR_+^{\infty}\times \RR_+^{\infty})$, to $({\boldsymbol{X}}_{\uparrow}, {\boldsymbol{Z}})$.
\end{theorem}
This theorem is proved in Section \ref{sec:origdyn}.
\begin{remark}
The condition \eqref{condnn} will be satisfied in a variety of common scenarios. For example, \eqref{condnn} holds whenever there exists a deterministic sequence
\(h_i\uparrow\infty\) such that \(\sum_i e^{-c h_i}<\infty\) for every
\(c>0\), and such that, for every $\varepsilon>0$, there exist $\gamma>0$ and $N_0\in\N$ for which
\[
    \liminf_{n\to\infty}
    \mathbb{P}\left(
        \widehat X_i^n(0) \ge \gamma h_i
        \text{ for all } i\ge N_0,\ i\le n
    \right)
    \ge 1-\varepsilon .
\]
This includes linear, polynomial, and super-logarithmic lower envelopes,
for instance \(h_i=i^\alpha\), \(\alpha>0\), or
\(h_i=(\log i)^{1+\eta}\), \(\eta>0\). In particular, it will hold when the associated diffusion scaled gaps are distributed as $\pi^{(n)}_{(a,b)}$ defined in Theorem \ref{jackcor}, for any $a>b>0$, which is an approximate stationary distribution for $\hat{\boldsymbol{Z}}^n$ (see Corollary \ref{cor: mix}). 
\end{remark}

Along the lines of \cite[Remark 2.3]{banerjee2024extremal} it can be shown that the gap process $\boldsymbol{Z} = (Z_1, Z_2, \ldots)$ defined by \eqref{eq:hatxhatz} is a continuous Markov process with state space
$$\Lambda := \{\boldsymbol{z} = (z_1, z_2, \ldots) \in \RR_+^{\infty}: \boldsymbol{x}(\boldsymbol{z}) := (z_1, z_1+ z_2, \ldots) \in \Lambda_0\}.$$
Our next result gives  stationary distributions for this Markov process.
This theorem will be proved in Section \ref{sec:invmzrpf}.
For $\theta>0$, let $Exp(\theta)$ be the probability distribution of an Exponential random variable with mean $\theta^{-1}$.
\begin{theorem}
\label{stationarydistreflectedatlasresult}
    Consider the  reflected infinite Atlas model with drift $\delta=b$. Then, for any $a \in (b,\infty)$ the product-form distribution 
    \begin{equation}\label{eq:muabdef}
    \mu_{a,b} := Exp(a-b) \otimes Exp(a)\otimes Exp(a) \cdots
    \end{equation}
    is a stationary distribution for the Markov process $\boldsymbol{Z}$.
\end{theorem}

Next, we would like to relate the long time behavior of the process $\hat{\boldsymbol{Z}}^n$ associated with parameters $(a,b)$ with the stationary distribution $\mu_{a,b}$.  For this, we will consider the slight modification of the underlying dynamics, that was referred to as the {\em system with pauses} in Section \ref{sec:intro}.  In order to describe this process with modified dynamics, we first give a distributionally equivalent representation for the process
$(X_{(1)}^n(t), \ldots , X_{(n)}^n(t))$. This representation will be given in terms of another collection $(\tilde\cla_i, \tilde\cld_i, i \in [n])$ of mutually independent Poisson random measures (PRM) on $\RR_+\times \RR_+$ with intensity  as the Lebesgue measure $dt\times dr$ on $\RR_+\times \RR_+$.
As before, suppose that the initial queue-length vector is given as $\{x^n_i, i \in [n]\}$ and that the initial vector is ordered, namely $x^n_1 \le \cdots \le x^n_n$.
Define, for $i \in [n]$,
$$\gamma^n_i :=
\begin{cases}
n-an^{-1/2}, \;\; i \in [n]\setminus\{1\}\\
n-an^{-1/2} + bn^{1/2}, \;\; i =1.
\end{cases}
$$
Then the following representation gives a process  that has the same distribution as the process $(X_{(1)}^n, \ldots , X_{(n)}^n)$ introduced below \eqref{eq:xnit} (we abuse notation and use the same symbols to denote this process). See \cite[Section 2]{banbudest}.

\begin{align*}
X_{(i)}^n(t) &= x^n_i +
\int_{[0,t]\times \RR_+} 1_{[0, \gamma^n_i + \sum_{j=1}^{i-1} 
\gamma^n_j1\{X^n_{(i)}(s-)=X^n_{(j)}(s-)\}]}(r)1_{\{X^n_{(i)}(s-)<X^n_{(i+1)}(s-)\}} \tilde \cla_i(ds\, dr)\\
&\quad - \int_{[0,t]\times \RR_+} 1_{[0, n + n\sum_{j=i+1}^n 
1\{X^n_{(i)}(s-)=X^n_{(j)}(s-)\}]}(r)1_{\{X^n_{(i)}(s-)>X^n_{(i-1)}(s-)\}} \tilde \cld_i(ds\, dr),
\end{align*}
where, by convention, we take $X^n_{(0)}(t)=0$ and $X^n_{(n+1)}(t)=\infty$
for all $t\ge 0$.

We now introduce the process with modified dynamics.  Recall that, in this system with pauses, when there is a tie that involves the ranks $i, i+1, \ldots, i+j$, departures from queues $i+1,...,i+j$ and arrivals to queues $i,i+1,...,i+j-1$ are paused.
Thus, jobs can only arrive to the highest queue in the tie and depart from the lowest queue in the tie.  Denoting the ordered queue-lengths in this modified process as $(Y^n_1(t), \ldots , Y^n_n(t))$ we can give a representation for the evolution for the process in terms of PRM $(\tilde\cla_i, \tilde\cld_i, i \in [n])$ as follows.
\begin{equation}\label{eq:yeqns}
\begin{aligned}
Y_i^n(t) &= x^n_i +
\int_{[0,t]\times \RR_+} 1_{[0, \gamma^n_i]}(r)1_{\{Y^n_{i}(s-)<Y^n_{i+1}(s-)\}} \tilde \cla_i(ds\, dr)\\
&\quad - \int_{[0,t]\times \RR_+} 1_{[0, n]}(r)1_{\{Y^n_{i}(s-)>Y^n_{i-1}(s-)\}} \tilde \cld_i(ds\, dr),
\end{aligned}
\end{equation}
where again by convention, we take $Y^n_{0}(t)=0$ and $Y^n_{n+1}(t)=\infty$
for all $t\ge 0$.
Let $U^n_i := Y^n_i-Y^n_{i-1}$, for $i\in [n]$.
As before, we consider the diffusion scaled processes:
$$\hat Y^n_i(t) = \frac{Y^n_i(t)}{\sqrt{n}}, \; \hat U^n_i(t) = \frac{U^n_i(t)}{\sqrt{n}}, \;\; t \ge 0.$$

Our next theorem gives a result analogous to Theorem \ref{convergencethm} for the processes $\hat{\bs{Y}}^n = (\hat Y^n_1, \ldots , \hat Y^n_n)$ and $\hat{\bs{U}}^n = (\hat U^n_1, \ldots , \hat U^n_n)$ and show that the processes $(\hat{\boldsymbol{X}}^n_{\uparrow}, \hat{\boldsymbol{Z}}^n)$ and
$(\hat{\boldsymbol{Y}}^n, \hat{\boldsymbol{U}}^n)$ have the same limit law.
\begin{theorem}
\label{convergencethm2}
Suppose that the initial queue-length vector $\{\hat{X}^n_i(0), i \in [n]\}$ converges in distribution to some $\mu \in \clp(\Lambda_0)$, and $\bs{X}, {\boldsymbol{X}}_{\uparrow}$ and ${\boldsymbol{Z}}$ are as in Theorem \ref{convergencethm}.
Then, as $n\to \infty$,
$(\hat{\boldsymbol{Y}}^n, \hat{\boldsymbol{U}}^n)$ converges in distribution, in $D([0,\infty), \RR_+^{\infty}\times \RR_+^{\infty})$, to $({\boldsymbol{X}}_{\uparrow}, {\boldsymbol{Z}})$.
\end{theorem}
This theorem is proved in Section \ref{proofofdiffusionlimitsect2}.

Next, we study convergence of stationary distributions of the modified diffusion scaled gap process $\{\hat U^n_i, i \in [n]\}$. 

For this, we consider a more general rank-based routing scheme where arrivals to the $i$-th ranked queue happen at rate $\lambda_i$ and departures from it happen at rate $\mu_i$. During ties, arrivals and departures occur in accordance with the system with pauses, as described before. The gaps between the successive ranked queues in this system (with the shortest queue taken as the first `gap') have an exact description in terms of a special type of (open) Jackson network as follows.
The state space of this Jackson network will be $\NN_0^k$ for some $k \in \NN,$ and the state process will be denoted as
$\bs{Q} = (Q^1, \ldots Q^k)$. The network will be defined in terms of positive parameters $\{\lambda_i\}_{i\in [k]}$
and $\{\mu_i\}_{i\in [k]}$. The external arrival rates and service rates in this Jackson network will be denoted as
$\{\bar \lambda_i\}_{i\in [k]}$ and $\{\bar \mu_i\}_{i\in [k]}$, with 
\[
\bar\lambda_i = 0 \quad i\in [k-1], \qquad \bar\lambda_k = \lambda_k,
\]
and
\[
\bar\mu_1 = \mu_1, \qquad
\bar\mu_i = \lambda_{i-1} + \mu_i, \quad i \in [2,k].
\]
The routing probabilities matrix will be denoted as $\{P_{ij}\}_{i,j \in [k]}$, with
\begin{align*}
P_{1,2} = 1, \;\;& P_{k,k-1} = \frac{\lambda_{k-1}}{\lambda_{k-1}+\mu_k}
\\
P_{i,i-1} = \frac{\lambda_{i-1}}{\lambda_{i-1}+\mu_i}, \;\; &
P_{i,i+1} = \frac{\mu_i}{\lambda_{i-1}+\mu_i},
\quad i \in \{2,..., k-1\},
\end{align*}
with all other $P_{ij}=0$.

Thus in this Jackson network, external arrivals only occur at the $k$-th queue and jobs after completing service at the $i$-th queue, $i\in \{2,...,k-1\}$ can only go to either the $(i-1)$-th or the $(i+1)$-th queue, while for the first queue the service completions are routed with probability $1$ to the second queue. Finally for the $k$-th queue, a completed job is routed to the $(k-1)$-th queue with probability $\frac{\lambda_{k-1}}{\lambda_{k-1}+\mu_k}$ while it leaves the system with
probability $\frac{\mu_{k}}{\lambda_{k-1}+\mu_k}$.

To see the connection between this Jackson network and the process $\bs{U}^n = (U^n_1, \ldots , U^n_n)$ of interest here, note that, from \eqref{eq:yeqns}, the evolution equation for $\bs{U}^n$ can be written as
\begin{equation*}
\begin{aligned}
U_i^n(t) &= (x^n_i- x^n_{i-1}) +
\int_{[0,t]\times \RR_+} 1_{[0, \gamma^n_i]}(r)1_{\{U^n_{i+1}(s-)>0\}} \tilde \cla_i(ds\, dr)\\
&\quad + \int_{[0,t]\times \RR_+} 1_{[0, n]}(r)1_{\{U^n_{i-1}(s-)>0\}} \tilde \cld_{i-1}(ds\, dr)
- \int_{[0,t]\times \RR_+} 1_{[0, n]}(r)1_{\{U^n_{i}(s-)>0\}} \tilde \cld_i(ds\, dr)\\
&\quad -\int_{[0,t]\times \RR_+} 1_{[0, \gamma^n_{i-1}]}(r)1_{\{U^n_{i}(s-)>0\}} \tilde \cla_{i-1}(ds\, dr)
,
\end{aligned}
\end{equation*}
We can now identify the system $\bs{U}^n$ with the above Jackson network by taking $k=n$, $\lambda_i = \gamma^n_i$, and $\mu_i=n$, for $i\in [n]$.

We now give a result for the stationary distribution of the Jackson network $\bs{Q}$ which is a consequence of classical facts about Jackson networks\cite{Kelly1979}.
\begin{theorem}\label{thm:jack}
$\bs{Q}$ has a unique  stationary distribution on $\NN_0^k$ if and only if
\[
\rho_i := \prod_{j=i}^k \frac{\lambda_j}{\mu_j} < 1,
\qquad 1 \le i \le k .
\]
In this case, the stationary distribution is
\[
\pi(\bs{q}) = \prod_{i=1}^k \rho_i^{q_i} (1-\rho_i),
\qquad \bs{q} = (q_1,\dots,q_k) \in \NN_0^k.
\]
\end{theorem}
As an immediate consequence of this theorem and the identification of $\bs{U}^n$ with a Jackson network of the above form, we now have the following result on convergence of stationary distributions of the modified diffusion scaled gap process $\{\hat U^n_i, i \in [n]\}$.
In what follows, we will identify a probability measure on $\RR_+^n$ with a probability measure on $\RR_+^{\infty}$ in the obvious manner. Specifically, if $V = (V_1, \ldots, V_n)$ is a $\RR_+^n$ valued random variable with probability law $\theta$ on $\RR_+^n$, we will denote the law of $(V_1, \ldots, V_n, 0, \ldots)$ on $\RR_+^{\infty}$, once more as $\theta$.
Let $\hat{\boldsymbol{U}}^n = (\hat U^n_1, \ldots , \hat U^n_n)$.
Recall the parameters $(a,b)$ that characterize the dynamics of $\hat{\boldsymbol{U}}^n$.
\begin{theorem}\label{jackcor}
For each $n$, the Markov process $\hat{\boldsymbol{U}}^n$ has a unique stationary distribution $\pi^{(n)} = \pi^{(n)}_{(a,b)}$ on $\NN_0^n/\sqrt{n}$, given as 
$$\pi^{(n)}([z_1, \infty)\times \cdots \times [z_n, \infty)) = \prod_{i=1}^n
\left(\rho_i^{(n)}\right)^{\sqrt{n}z_i}, \;
z_i \in \NN_0/\sqrt{n}, \; i \in [n],
$$
where 
\[
\rho_i^{(n)}
=
\prod_{j=i}^n
\left(1-\frac{a}{n^{3/2}}+\frac{b}{\sqrt{n}}\mathbf{1}_{\{j=1\}}\right), \; i \in [n].
\]
Furthermore, regarding $\pi^{(n)}$ as a probability measure on $\RR_+^{\infty}$, $\pi^{(n)} \to \mu_{a,b}$, weakly, where $\mu_{a,b}$ is as in Theorem \ref{stationarydistreflectedatlasresult}.
\end{theorem}
We note the following important corollary of the above result and Theorem \ref{convergencethm}.
Denote the coordinate sequence on $(\RR_+^{\infty}, \clb(\RR_+^{\infty}))$ as $\bs{\omega} = (\omega_1, \omega_2, \ldots).$
Let $\tilde \om_i = \om_{i}- \om_{i-1}$, $i \in \NN$, with $\om_0:=0$, and let $\tilde{\bs{\om}} = (\tilde\omega_1, \tilde\omega_2, \ldots)$. 
With $\mu_{a,b}$ as in \eqref{eq:muabdef}, let $\nu_{a,b}$ be the probability measure on $(\RR_+^{\infty}, \clb(\RR_+^{\infty}))$ under which $\boldsymbol{\omega}\sim \nu_{a,b}$ 
when the associated 
$\tilde{\bs{\om}}\sim \mu_{a,b.}$ Also, and with analogous notation of coordinate sequence as above, let $\tilde \pi^{(n)}_{(a,b)}$ be the probability measure on $(\RR_+^{n}, \clb(\RR_+^{n}))$, under which $\boldsymbol{\omega}\sim \tilde \pi^{(n)}_{(a,b)}$ when $\tilde{\bs{\omega}}\sim \pi^{(n)}_{(a,b)}.$ 
Then we have the following corollary for the system under the original dynamics.
\begin{corollary}
\label{cor: mix}
Consider the Markov process $\hat{\boldsymbol{X}}^n_{\uparrow}$ started with the initial distribution 
$\tilde \pi^{(n)}_{(a,b)}$. Then this process has the following approximate stationarity property: For each fixed $T<\infty$, $k \in \NN$, and $f\in C_b(\RR_+^k)$,
$$
\limsup_{n\to \infty} \sup_{s,t \in [0,T]} \left| \EE f(\hat X^n_{(1)}(t), \ldots , \hat X^n_{(k)}(t))
- \EE f(\hat X^n_{(1)}(s), \ldots , \hat X^n_{(k)}(s))\right| =0.$$
Furthermore, as $n\to \infty$,
$(\hat{\boldsymbol{X}}^n_{\uparrow}, \hat{\boldsymbol{Z}}^n)$ converges in distribution, in $D([0,\infty), \RR_+^{\infty}\times \RR_+^{\infty})$, to the stationary process $({\boldsymbol{X}}_{\uparrow}, {\boldsymbol{Z}})$, where  $\boldsymbol{X}$  is the weak solution of  \eqref{eq:unranksde} with initial distribution $\nu_{a,b}$ and with $\delta=b$.
\end{corollary}

The following corollary to Theorem \ref{jackcor} provides asymptotic formulas for some key  system statistics under the approximate stationary distribution $\tilde \pi^{(n)}_{(a,b)}$.

\begin{corollary}\label{corstat}
    Consider the Markov process $\hat{\boldsymbol{X}}^n_{\uparrow}$ started with the initial distribution 
$\tilde \pi^{(n)}_{(a,b)}$. Then, we have the following.
\begin{itemize}
    \item[(i)] For any fixed $k \in \mathbb{N}$, the $k$-th ranked queue satisfies
    $$
\frac{\EE[X_{(k)}^{n}(0)]}{\sqrt n\left(\frac{1}{a-b} + \frac{k-1}{a}\right)}  \rightarrow 1, \ \text{ as } \ n \rightarrow \infty.
    $$
    \item[(ii)] The asymptotics for the {\em system imbalance index} (cf. \cite{banbudest}) are given as
    $$
\frac{\EE[X_{(n)}^{n}(0) - X_{(1)}^{n}(0)]}{n^{3/2}a^{-1} \log n} \rightarrow 1, \ \text{ as } \ n \rightarrow \infty.
    $$
    \item[(iii)] The average queuelength satisfies
    $$
    \frac{\EE\left[\frac{1}{n}\sum_{k=1}^n X_{(k)}^{n}(0)\right]}{n^{3/2}a^{-1}} \rightarrow 1, \ \text{ as } \ n \rightarrow \infty.
    $$
\end{itemize}
\end{corollary}

\begin{remark}\label{rem:qual}
We now use Corollary \ref{corstat} to compare the performance of the MJSQ scheme considered here with two basic policies in the many server regime.

\textbf{(i) Random Routing (RR): } Suppose that the total arrival rate is $n^2 - (a-b)\sqrt{n}$ and service rate is $n$ at each server, and all the jobs are routed at random. Then, it can be checked from the fact that the stationary queuelengths in RR are iid Geometric random variables on $\NN_0$ with success probability $(a-b)n^{-3/2}$, that the expected length of any given queue (and consequently, the expected average queuelength), in steady state, is exactly $\frac{n^{3/2}}{(a-b)}-1$. Moreover, the expected queuelength of the $k$-th ranked queue, scaled by $k\sqrt{n}/(a-b)$, approaches $1$ as $n \to \infty$, and the system imbalance index grows like $n^{3/2} (a-b)^{-1}\log n$. The first part of the last statement can be checked by observing that the $k$-th ranked queuelength, scaled by $\sqrt{n}$, weakly converges to a Gamma distribution with shape parameter $k$ and rate parameter $a-b$, and verifying a uniform integrability property. The second part  of the statement follows from calculations involving order statistics of iid Geometric random variables. Thus, when $a-b \ll a$ (the system is very heavily loaded), load is balanced much more effectively under the MJSQ scheme considered here, and the average queuelength is also significantly smaller. Furthermore, although the shortest queue gets a much larger stream of arrivals routed to it under the MJSQ scheme than under RR,
we find that it has the identical first-order asymptotic behavior under the two schemes. Thus, if `priority customers' are routed to the shortest queue via the Poisson stream of rate $b\sqrt{n}$, their waiting time is roughly the same as RR, and there is no added congestion at the shortest queue due to this routing.

\textbf{(ii) Join-the-shortest-queue (JSQ): }
Since we are interested in comparison at steady state, we can equivalently consider a system that is slowed down by a factor of $n$. With that in mind,
consider the JSQ policy for a system consisting of $n$ servers, each processing jobs at rate $1$. Jobs arrive at a central dispatcher at rate $n - (a-b)n^{-1/2}$, and are then sequentially routed to the server with the shortest queue. The steady state of this system  serves as a direct comparison to that of the MJSQ routing scheme considered here. Many server heavy traffic schemes for JSQ give rise to hard and interesting problems which form an active area of research. In particular, $n$-server JSQ systems with net arrival rate $n - \lambda n^{1-\alpha}$, where $\alpha,\lambda>0$, have been extensively studied. When $\alpha=1/2$, this corresponds to the standard \emph{Halfin-Whitt regime} \cite{HW}, where most servers have queue-length one and rescaled fluctuations from this `all length one queues' system converge weakly to a reflected singular diffusion process with interesting long time behavior \cite{eschenfeldt2018join,Banerjee_2019,Banerjee_2020,Anton20}. Other well-studied regimes are the \emph{sub-Halfin-Whitt ($\alpha \in (0,1/2)$) \cite{liu2020steady}, super-Halfin-Whitt ($\alpha \in (1/2,1)$) \cite{liu2021universal,zhao2025many}, non-degenerate slowdown ($\alpha = 1$) \cite{atar2012diffusion,GW18}, and the super slowdown \cite{Hurtado-Lange2022,raj2024exponential} ($\alpha>1$)} regimes. 

The regime in the current work corresponds to $\alpha=3/2$. To the best of our knowledge, obtaining suitable diffusion approximations and a full picture of the steady state behavior of the JSQ system when $\alpha >1$ (in particular, $\alpha=3/2$) is an open problem. It has been conjectured that both the minimum and the average steady state queuelengths for the JSQ system scale like $n^{\alpha-1}$ when $\alpha>1$. This has been resolved for the average queuelength for $\alpha>1$ by \cite{raj2024exponential}, who obtained a large deviations result for this quantity. In comparison, for the MJSQ system considered here, Corollary \ref{corstat} shows that the average queuelength scales like $n^{3/2}$. Thus, the price we pay for the significantly reduced communication cost of MJSQ is in higher average queuelengths. However, based on preliminary calculations, we conjecture that MJSQ and JSQ deliver comparable efficiency at the level of the $k$ shortest queues for any fixed $k$. 
In future work, we will consider schemes that interpolate between JSQ and MJSQ  by suitably scaling up the proportion of incoming jobs sent to the shortest queue in the MJSQ policy. This may provide a systematic approach for analyzing the challenging regime $\alpha >1$, while also shedding light on the tradeoff between load-balancing performance and communication overhead.

\end{remark}

As another corollary of Theorem \ref{thm:jack}, we see that for the {\em Marginal Size Bias Load-Balancing policies} studied in 
\cite{banbudest} with a fixed number of queues $k$, the stationary gap distribution of the associated system with pauses, under the heavy traffic limit, converges to the same distribution as for the original system. This distribution corresponds to the unique stationary law for the finite-dimensional reflected diffusion process obtained as the (transient) heavy traffic limit, and the `interchange of limits'  for the original system that established the convergence of stationary distributions to this limiting distribution was proved in \cite[Theorem 3]{banbudest}. 
\begin{corollary}
\label{cor:mjsq}
Consider the Jackson network of the above form  with rates depending on $n$, namely,   $\lambda_i \equiv \lambda_i^{(n)} = n - a_i \sqrt{n}$ and $\mu_i \equiv \mu_i^{(n)} = n$ for $i\in [k]$, where $a_i$ are real numbers such that
 $\sum_{j=i}^k a_j > 0$ for all $1 \le i \le k$.
Then, $\bs{Q}$ has a unique stationary distribution. Denoting the unique stationary distribution of $\bs{Q}^n := \bs{Q}/\sqrt{n}$ as $\pi^{(n)}$,
$\pi^{(n)}$ converges weakly to $\bigotimes_{i=1}^k Exp\left(\sum_{j=i}^k a_j\right)$.
\end{corollary}

\begin{remark}\label{rem:compare}
Corollary \ref{cor:mjsq} says that, when the number of queues is fixed, the long time behavior of the gap process for the original system and the system with pauses asymptotically coincide as $n \to \infty$. In future work we aim to establish a similar result when $k(n)=n$, however currently there are significant obstacles in proving such a result. A natural `synchronous coupling' of the two systems, constructed from the same Poisson random measures encoding arrivals and departures, that one would like to use to control the discrepancy between the two systems, lacks desirable monotonicity properties. Furthermore, it can be shown that the discrepancy between the two systems under this coupling, which builds up at times when there is a tie in either system, is unbounded over the infinite time horizon. Thus, one needs appropriate \emph{relaxation time} estimates (rates at which the processes approach equilibrium), and show that the time spent at ties for either system up to the relaxation time is asymptotically negligible, as $n \to \infty$.
\end{remark}

\color{black}

Theorems \ref{thm:jack}, \ref{jackcor} and Corollaries \ref{cor: mix}, \ref{corstat} and \ref{cor:mjsq}  are proved in Section \ref{convergenceofstationariessect}.

\section{Wellposedness of reflected Atlas model}
\label{sec:thm1pf}
In this section we prove Theorem \ref{uniquenessthm}. The proof follows the ideas of \cite{pitmanpal}.
Fix $\mu \in \clp(\Lambda_0)$ as in the statement of the theorem.
Let $(\Om, \clf_0, \QQ)$ be a probability space on which are given mutually independent real standard Brownian motions $\{B_i, i \in \NN\}$ and a $\RR_+^{\infty}$ valued random variable $\bs{X}(0) = (X_1(0), X_2(0), \ldots)$ with distribution $\mu$ that is independent of $\{B_i, i \in \NN\}$.
Let $$\clf^0_t = \sigma\{\bs{X}(0), B_i(s), 0 \le s \le t, i \in \NN\}, \; \clf_t = \bar{\clf}^0_t, \;\; 
\clf = \sigma\{\cup_{t>0}\clf_t\}$$
where $\bar{\clf}^0_t$ is the augmentation of $\clf^0_t$ under $\QQ$.
Let $\Gamma$ be the one dimensional Skorohod map on $C([0,\infty), \RR)$, namely for $f \in C([0,\infty), \RR)$
$$\Gamma(f)(t) = f(t) - \inf_{0\le s \le t} (f(s) \wedge 0), \; t \ge 0.$$
Let
$$X_i(t) := \Gamma(X_i(0) + \sqrt{2}B_i(\cdot))(t), \; t \ge 0.
$$
Then $X_i$ defines a reflected Brownian motion (RBM) on $\RR_+$ started from $X_i(0)$ and these RBM are mutually independent (conditioned on $\bs{X}(0)$).
The following lemma says that for each $t$, a.s., the vector $\{X_i(t), i \in \NN\}$ can be ordered in a unique manner.
\begin{lemma}
There exist jointly measurable adapted bijections $\pi_t: \NN \to \NN$, $t \ge 0$, such that, for any $i,j \in \NN$, $\pi_t(i) < \pi_t(j)$ if and only if $X_i(t) < X_j(t)$ or $X_i(t)=X_j(t)$ and $i <j$. Writing $X_{(i)}(t) := X_{\pi_t(i)}(t)$, we have $X_{(1)}(t) \le X_{(2)}(t) \le \cdots$.
\end{lemma}
\begin{proof}
Fix $\bs{x} \in \Lambda_0$ and $t>0$. Then, for any $i \in \NN$ and $c\le x_i$ by the Gaussian bound we have
\begin{align*}
\QQ(\inf_{0\le s \le t} (x_i+\sqrt{2}B_i(s)) \le c) \le 2\exp\{-(x_i-c)^2/4t\}.
\end{align*}
Recalling the definition of $\Lambda_0$, we now conclude from the Borel-Cantelli lemma that for any $c \in \RR$,
$$\QQ(\inf_{0\le s \le t} (x_i+\sqrt{2}B_i(s)) \le c, \mbox{i.o.})=0.$$
By a conditioning argument this shows that
$$\QQ(\inf_{0\le s \le t} (X_i(0)+\sqrt{2}B_i(s)) \le c, \mbox{i.o.})=0.$$
The result is now immediate.

\end{proof}
\begin{lemma}
For each $t \ge 0$, 
\begin{equation}\label{eq:nt}
N_t:= \frac{1}{\sqrt{2}}\sum_{j=1}^{\infty} \int_0^t 1(X_j(s) = X_{(1)}(s)) dB_j(s)
\end{equation}
is well defined. Furthermore,
$\{N_t, t \ge 0\}$ is a continuous square integrable $\clf_t$-martingale under $\QQ$ with  quadratic variation $\langle N\rangle_t =t/2$, $t\ge 0$.
Finally, 
$$D_t = \exp\left(\delta N_t - \delta^2t/4\right),\;\; t \ge 0$$
is a $\clf_t$-martingale.
\end{lemma}
\begin{proof}
The proof is an immediate consequence of the identity
$$\sum_{j=1}^{\infty} \int_0^t 1(X_j(s) = X_{(1)}(s)) ds = t, \; t \ge 0$$
which follows from the fact that for any $i\neq j$
$$\int_0^t 1(X_j(s) = X_i(s)) ds = 0, \; t \ge 0, \; \mbox{ a.s. }$$
\end{proof}

We can now complete the proof of Theorem \ref{uniquenessthm}.\\

\noindent {\em Proof of Theorem \ref{uniquenessthm}.}
Define the probability measure $\PP$ on $(\Om, \clf)$ by setting
$$\PP(A) \doteq \EE_{\QQ}(D_t 1_A), \; A \in \clf_t, \; t \ge 0.$$
By Girsanov's theorem it then follows that, letting  $\tilde B_i(t) := B_i(t) - \frac{1}{\sqrt{2}}\delta\int_0^t 1(X_i(s) = X_{(1)}(s)) ds$, $t \ge 0$, $i \in \NN$,
$\{\tilde B_i\}$ is a collection of Brownian motions on the filtered probability space $(\Om, \clf, \PP, \{\clf_t\}_{t\ge 0})$.
Writing
\begin{align*}
X_i(t) = \Gamma\left(X_i(0) +  \delta \int_0^{\cdot} 1(X_i(s) = X_{(1)}(s)) ds + \sqrt{2}\tilde B_i(\cdot)\right)(t), \; t \ge 0, \; i \in \NN,
\end{align*}
we see that $\boldsymbol{X}= (X_1, X_2, \ldots)$ is a weak solution of \eqref{eq:unranksde},  associated with the Brownian motions  $\{\tilde B_i, i \in \NN\}$, on the filtered probability space $(\Om, \clf, \PP, \{\clf_t\}_{t\ge 0})$.

Now we prove uniqueness. Let $\boldsymbol{X}= (X_1, X_2, \ldots)$ be a weak solution of \eqref{eq:unranksde} on some filtered probability space $(\Om, \clf, \PP, \{\clf_t\}_{t\ge 0})$ associated with  Brownian motions  $\{B_i, i \in \NN\}$.
Define $\{N_t\}$ by \eqref{eq:nt}. Then with $\clg_t := \sigma\{X_i(s), B_i(s), \; i \in \NN, 0 \le s \le t\}$, $\{N_t\}$ is continuous, square integrable  $\clg_t$-martingale with $\langle N\rangle_t = t/2$.
Further
$$\tilde D_t := \exp\left(-\delta N_t - \delta^2t/4\right),\;\; t \ge 0,
$$
 is a $\{\clg_t\}$-martingale as well. Let $\clg := \sigma(\cup_{t\ge 0} \clg_t)$ and define the probability measure $\tilde \QQ$ on $(\Om, \clg)$ as
 $$\tilde \QQ(A) \doteq \EE_{\PP}(\tilde D_t 1_A), \; A \in \clg_t, \; t \ge 0.$$
 It then follows that for any $k \in \NN$, $0 \le t_1 \le t_2 \cdots \le t_k \le t$, $t\ge 0$, $A \in \clb(\RR_+^k)$, with
 $B= \{(X_{t_1}, \ldots, X_{t_k}) \in A\}$,
 \begin{align*}
 \PP(B) &= \EE_{\tilde \QQ}\left( 1_B (\tilde D_t)^{-1}\right)\\
 &= \EE_{\tilde \QQ}\left( 1_{\{(X_{t_1}, \ldots, X_{t_k}) \in A\}} 
 \exp\left(\frac{\delta}{\sqrt{2}} \sum_{j=1}^{\infty} \int_0^t 1(X_j(s) = X_{(1)}(s)) dB_j(s)+ \frac{\delta^2t}{4}\right)\right)\\
 &=\EE_{\tilde \QQ}\left( 1_{\{(X_{t_1}, \ldots, X_{t_k}) \in A\}} 
 \exp\left(\frac{\delta}{\sqrt{2}} \sum_{j=1}^{\infty} \int_0^t 1(X_j(s) = X_{(1)}(s)) d \hat B_j(s)
 - \frac{\delta^2t}{4}\right)\right),
 \end{align*}
 where $\hat B_i(t) := B_i(t) + \frac{1}{\sqrt{2}}\delta\int_0^t 1(X_i(s) = X_{(1)}(s)) ds$, $t \ge 0$, $i \in \NN$.

 The uniqueness now follows on noting that the last term in the above display is invariant under the choice of the weak solution since
 $\{\hat B_j, j \in \NN\}$ are Brownian motions on the filtered probability space $(\Om, \clg, \tilde\QQ, \{\clg_t\}_{t\ge 0})$,  and, under $\tilde \QQ$, $\bs{X}(0)$ has distribution $\mu$, and 
 $$
 X_j(t) = \Gamma(X_j(0) + \sqrt{2}\hat B_j(\cdot))(t), \; t \ge 0, j \in \NN,
 $$
 define mutually independent reflected Brownian motions starting from initial conditions $X_j(0)$, $j \in \NN$.
 \hfill \qed
 
\section{Convergence to the reflected Atlas model}
\label{sec:cogceproof}
In this section we prove Theorems \ref{convergencethm} and \ref{convergencethm2}. 

\subsection{Proof of Theorem \ref{convergencethm}.}
\label{sec:origdyn}
The main idea is to couple the system \eqref{eq:xnit} with $n$-parallel queues with a related system with a fixed number $N$ of queues. This system is defined using the same initial condition  and  PRM as used in constructing the system in \eqref{eq:xnit}. The system is once more indexed by the heavy traffic parameter $n$, and in the $n$-th system there are $N$ parallel servers, the arrival rate to the shortest queue (where ties in queue length are broken in lexicographical order) is
$n - \frac{a}{\sqrt{n}} + b \sqrt{n},$
and the arrival rate to the remaining $N-1$ queues will be
$n- \frac{a}{\sqrt{n}}$. The state evolution, with $X^{n,N}_i$ denoting the queue-length process for the $i$-th server in this system,  is given as
\begin{equation}
\label{eq:xnitN}
X^{n,N}_i(t) = X^n_i(0) + A^{n,N}_i(t) - D^{n,N}_i(t), \; i \in [N], t\ge 0,
\end{equation}
where
\begin{align*}
A^{n,N}_i(t) &= \int_{[0,t]\times \RR_+} 1_{[0, n - an^{-1/2} + bn^{1/2} 1_{\{X^{n,N}_i(s-) = X^{n,N}_{(1)}(s-)\}}]}(r) \cla_i(ds\, dr), \; i \in [N],\\
D^{n,N}_i(t) &= \int_{[0,t]\times \RR_+} 1_{[0, n]}(r) 1_{\{X^{n,N}_i(s-) \neq 0\}} \cld_i(ds\, dr), \; i \in [N].
\end{align*}
The lemma below shows that for each fixed $K$, the first $K$-ranked queues in the original system are well approximated by those in the $(n,N)$-th system, when $N$ is large. Similar to the original system, we use the notation $\hat X^{n,N}_i$ and $\hat X^{n,N}_{(i)}$ to denote the diffusion scaled queuelength processes and the corresponding ranked queuelength processes.
\begin{lemma}
\label{comparisonlemmaR}
    Let the initial queuelengths $\{X^n_i(0), i \in [n]\}$ and limit measure $\mu$ be as in the statement of Theorem \ref{convergencethm}. Fix  $T>0$ and $K \in \NN$. 
Then there exists a sequence $\{\epsilon_{N}, N \in \NN\}$ such that $\epsilon_{N}\rightarrow 0$ as $N\to \infty$, and
\begin{equation}
\label{comparisoneqnR}
    \liminf_{n \rightarrow \infty} P\Big((\hat{X}_{(1)}^{n}(t),...,\hat{X}^{n}_{(K)}(t))=(\hat{X}_{(1)}^{n,N}(t),...,\hat{X}_{(K)}^{n,N}(t)) \hspace{5mm} \forall t \in [0,T]\Big) \geq 1-\epsilon_{N}.
\end{equation}

\end{lemma}
\begin{proof}
We assume without loss of generality that $n>N > K+1$. Take $N_1 \in \NN$ such that $N> N_1 >K$. For notational convenience, we write $x_j:= \hat X^{n}_j(0)$. We note that, on the set
$$
A := \{\sup_{i \le K} \sup_{0\le t \le T} \hat X^n_i(t) < x_{N_1} \mbox{ and }
\inf_{i \ge N} \inf_{0\le t \le T} \hat X^n_i(t) > x_{N_1}\},
$$
for each $t \in [0,T]$ and $i \in [K]$, there is a $j \in [N]$ such that
$$\hat X^n_{(i)}(t) = \hat X^n_j(t) = \hat X^{n,N}_j(t) = \hat X^{n,N}_{(i)}(t).$$
Thus, letting
$$
\tau^{N,N_1,n}:= \inf\{t\ge 0: \sup_{i \le K}  \hat X^n_i(t) \ge x_{N_1} \mbox{ or }
\inf_{i \ge N}  \hat X^n_i(t) \le x_{N_1}\},
$$
we see that
$$
\liminf_{n\to \infty}
P\Big((\hat{X}_{(1)}^{n}(t),...,\hat{X}^{n}_{(K)}(t))=(\hat{X}_{(1)}^{n,N}(t),...,\hat{X}_{(K)}^{n,N}(t)) \hspace{5mm} \forall t \in [0,T]\Big) \ge \liminf_{n\to \infty}
P(\tau^{N,N_1,n} \ge T).
$$
Thus to complete the proof, it suffices to show that, for each $\kappa_0>0$, there is a $N_1>K$, such that
\begin{equation}\label{eq:857}
\limsup_{N\to \infty} \limsup_{n\to \infty} P(\tau^{N,N_1,n} < T) \le \kappa_0.
\end{equation}
Note that,
\begin{align}\label{eq:taunn1n}
P(\tau^{N,N_1,n} < T) &\le \sum_{i=1}^K P\left(\sup_{0\le t \le T} \hat X^{n}_i(t) \ge x_{N_1}\right)
+ P\left(\inf_{i \in \{N,...,n\}}\inf_{0\le t \le T} \hat X^{n}_i(t) \le x_{N_1}\right).
\end{align}
Next,  letting 
$$
\mathring{D}^{n}_i(t) = \int_{[0,t]\times \RR_+} 1_{[0, n]}(r)  \cld_i(ds\, dr), \; i \in [n],
$$
for $i \in [K]$, since $\sup_{0\le t \le T}(\sqrt{n}x_i + A^{n}_i(t)- {D}^{n}_i(t)) \le 2 \sup_{0\le t \le T}
|\sqrt{n}x_i + A^{n}_i(t)- \mathring{D}^{n}_i(t)|$ (by the Lipschitz property of the one-dimensional Skorohod map),
letting $a^n_i(t) =  (n- an^{-1/2} + bn^{1/2}1_{\{X^{n}_i(t) = X^{n}_{(1)}(t)\}})$,
\begin{equation}\label{eq:8110}
\begin{aligned}
P\left(\sup_{0\le t \le T} \hat X^{n}_i(t) \ge x_{N_1}\right) &\le
P\left(\sup_{0\le t \le T}  |A^{n}_i(t)- \mathring{D}^{n}_i(t)| \ge \sqrt{n}(x_{N_1}-2x_i)/2\right)\\
&\le P\left(\sup_{0\le t \le T}  |A^{n}_i(t) - \int_0^t a^n_i(s) ds| \ge \sqrt{n}T\alpha^{N_1,n}_i\right)\\
&\quad + P\left(\sup_{0\le t \le T}  |\mathring{D}^{n}_i(t) - nt| \ge \sqrt{n}T\beta^{N_1,n}_i\right),
\end{aligned}
\end{equation}
where
$$
\alpha^{N_1,n}_i= \left((x_{N_1}-2x_i)/4T -b\right), \; \beta^{N_1,n}_i= (x_{N_1}-2x_i)/4T.
$$
Noting that, for $n$ sufficiently large, 
$\int_0^t a^n_i(s) ds \le 2nT$,
with $\{\cln(t)\}$ a rate $1$ Poisson process, for $i \in [K]$,
\begin{align*}
&P\left(\sup_{0\le t \le T}  |A^{n}_i(t) - \int_0^t a^n_i(s) ds| \ge \sqrt{n}T\alpha^{N_1,n}_i\right)\\
&\le P\left(\sup_{0\le t \le 2nT}  |\cln(t) -t| \ge \sqrt{n}T\alpha^{N_1,n}_i\right)\\
&\le P(\alpha^{N_1,n}_i < 2) + P\left(\sup_{0\le t \le \alpha^{N_1,n}_inT}  |\cln(t) -t| \ge \sqrt{n}T\alpha^{N_1,n}_i\right)\\
&\le P( \hat X^n_{N_1}(0) \le 2(\hat X^n_{K}(0) + 2T(2+b))) + P\left(\sup_{0\le t \le \alpha^{N_1,n}_inT}  |\cln(t) -t| \ge \sqrt{n}T\alpha^{N_1,n}_i\right).
\end{align*}
Fixing $\kappa:= \kappa_0/(20K)$ and using the weak convergence assumption on $\hat{\bs{X}}^n(0)$, choose 
$M \in (0,\infty)$ such that
\begin{equation}\label{eq:207}
P(2(\hat X^n_{K}(0) + 2T(2+b)) \ge M) \le \kappa \mbox{ for all } n \in \NN.
\end{equation}
Using \eqref{condnn} now choose $\tilde N_1 \in \NN$ such that for all $N_1 \ge \tilde N_1$, $\int_0^t a^n_i(s) ds \le 2nT$, and
\begin{equation}\label{eq:209}
\limsup_{n\to \infty} P(\hat X^n_{N_1}(0) \le M)=\limsup_{n\to \infty} P(e^{-\hat X^n_{N_1}(0)} \ge e^{-M}) \le \kappa.
\end{equation}
Then, for such $N_1$,
\begin{align*}
&\limsup_{n\to \infty}P\left(\sup_{0\le t \le T}  |A^{n}_i(t) - \int_0^t a^n_i(s) ds| \ge \sqrt{n}T\alpha^{N_1,n}_i\right)\\
&\le 2\kappa + \limsup_{n\to \infty} P\left(\sup_{0\le t \le \alpha^{N_1,n}_inT}  |\cln(t) -t| \ge \sqrt{n}T\alpha^{N_1,n}_i\right).
\end{align*}
From tail bounds for Poisson processes, we have that there exist $A_1, A_2, \tilde n_1>0$ such that for all $c>0$ and $n\ge \tilde n_1$
\begin{equation}\label{eq:poitail}
P\left(\sup_{0\le t \le nc}  |\cln(t) -t| \ge \sqrt{n}c\right) \le A_1\exp\{-A_2c\}.
\end{equation}
With $M$ as before and 
$\theta := \kappa \left( A_1 \exp\{A_2(bT + M/2)\}\right)^{-1}$, using \eqref{condnn} again, let $\tilde N_2 \ge \tilde N_1$ be  such 
that for all $N_1 \ge \tilde N_2$
\begin{equation}
\limsup_{n\to \infty} P(e^{-A_2\hat X^n_{N_1}(0)/4} \ge \theta) \le \kappa. \label{eq:210}
\end{equation}
Thus, for all $i \in [K]$, $N_1 \ge \tilde N_2$,
\begin{align*}
&\limsup_{n\to \infty} P\left(\sup_{0\le t \le \alpha^{N_1,n}_inT}  |\cln(t) -t| \ge \sqrt{n}T\alpha^{N_1,n}_i\right)\\
& \le \limsup_{n\to \infty} P\left(\hat X^n_K(0) \le M; \; \sup_{0\le t \le \alpha^{N_1,n}_inT}  |\cln(t) -t| \ge \sqrt{n}T\alpha^{N_1,n}_i\right) + \kappa\\
& \le \limsup_{n\to \infty} P\left(\hat X^n_K(0) \le M; e^{-A_2\hat X^n_{N_1}(0)/4} \le \theta;\;
\; \sup_{0\le t \le \alpha^{N_1,n}_inT}  |\cln(t) -t| \ge \sqrt{n}T\alpha^{N_1,n}_i\right) + 2\kappa\\
&\le 3 \kappa,
\end{align*}
where we have used \eqref{eq:207} in the second line, \eqref{eq:210} in the third line, and
\eqref{eq:poitail} in the last line.
Thus we have, for $N_1 \ge \tilde N_2$,
\begin{equation}\label{eq:817}
\limsup_{N\to \infty} \limsup_{n\to \infty}P\left(\sup_{0\le t \le T}  |A^{n}_i(t) - \int_0^t a^n_i(s) ds| \ge \sqrt{n}T\alpha^{N_1,n}_i\right) \le 5\kappa.
\end{equation}
Next,
\begin{align*}
P\left(\sup_{0\le t \le T}  |\mathring{D}^{n}_i(t) - nt| \ge \sqrt{n}T\beta^{N_1,n}_i\right)
&\le P(\beta^{N_1,n}_i<1) + P\left(\sup_{0\le t \le n\beta^{N_1,n}_iT}  |\cln(t) - t| \ge \sqrt{n}T\beta^{N_1,n}_i\right)
\end{align*}
With $M$ as before, we have, using \eqref{eq:207} and \eqref{eq:209},
$$\limsup_{n\to \infty}P(\beta^{N_1,n}_i<1) \le \limsup_{n\to \infty}[P(\hat X^n_{N_1}(0) \le M) + P(4T + 2\hat X^n_{K}(0) > M)] \le 2\kappa.$$
Also, with $\theta$ as before, for all $N_1 \ge \tilde N_2$, using \eqref{eq:207}, \eqref{eq:210}, and \eqref{eq:poitail}
\begin{align*}
&\limsup_{n\to \infty} P\left(\sup_{0\le t \le n\beta^{N_1,n}_iT}  |\cln(t) - t| \ge \sqrt{n}T\beta^{N_1,n}_i\right)\\
&\le \limsup_{n\to \infty} P\left(\hat X^n_K(0) \le M; e^{-A_2\hat X^n_{N_1}(0)/4} \le \theta;\;
\; \sup_{0\le t \le \beta^{N_1,n}_inT}  |\cln(t) -t| \ge \sqrt{n}T\beta^{N_1,n}_i\right) + 2\kappa\\
&\le 3\kappa.
\end{align*}
Thus, for $N_1 \ge \tilde N_2$
\begin{equation}\label{eq:816}
\limsup_{N\to \infty}\limsup_{n\to \infty} P\left(\sup_{0\le t \le T}  |\mathring{D}^{n}_i(t) - nt| \ge \sqrt{n}T\beta^{N_1,n}_i\right)
\le 5\kappa.
\end{equation}
Combining \eqref{eq:8110}, \eqref{eq:817}, \eqref{eq:816}, we have, for any $i \in [K]$, $N_1 \ge \tilde N_2$,
\begin{equation}\label{eq:818}
\limsup_{N\to \infty}\limsup_{n\to \infty}P\left(\sup_{0\le t \le T} \hat X^{n}_i(t) \ge x_{N_1}\right) \le 10\kappa.
\end{equation}
Now fix a $N_1 \ge \tilde N_2$. Since
$\{\hat X^n_{N_1}(0), n \in \NN\}$ is tight, there is a $M_1 \in \NN$ such that
$$\sup_{n\in \NN} P(\hat X^n_{N_1}(0) \ge M_1) \le \kappa.
$$
Then
\begin{align}\label{eq:822}
\limsup_{n\to \infty}P\left(\inf_{i \in \{N,...,n\}}\inf_{0\le t \le T} \hat X^{n}_i(t) \le x_{N_1}\right)
&\le \limsup_{n\to \infty} P\left(\inf_{i \in \{N,...,n\}}\inf_{0\le t \le T} \hat X^{n}_i(t) \le M_1\right) + \kappa.
\end{align}
Also, for $n$ sufficiently large,
\begin{equation}\label{eq:823}
\begin{aligned}
&P\left(\inf_{i \in \{N,...,n\}}\inf_{0\le t \le T} \hat X^{n}_i(t) \le M_1\right)
\le P\left(\sup_{0\le t \le T} |\mathring{D}^{n}_i(t)-A^{n}_i(t)| \ge \sqrt{n}(x_i-M_1) \mbox{ for some }
i \in \{N,...,n\}\right)\\
&\le P\left(\sup_{0\le t \le T} |A^{n}_i(t)- \int_0^t a^n_i(s) ds| \ge \sqrt{n}T \rho^n_i \mbox{ for some }
i \in \{N,...,n\}\right)\\
&\quad + P\left(\sup_{0\le t \le T} |\mathring{D}^{n}_i(t)- nt| \ge \sqrt{n}T \rho^n_i \mbox{ for some }
i \in \{N,...,n\}\right),
\end{aligned}
\end{equation}
where $\rho^n_i = \frac{1}{2}((x_i-M_1)/T - an^{-1})$.
Next, using \eqref{condnn}, choose $\tilde N\in \NN$ such that for all $N\ge \tilde N$, 
$$
\limsup_{n\to \infty} P(x_N \le M_1+ T(4+a)) \le \kappa.$$
Then, for all $N\ge \tilde N$, with $\cln_i$ mutually independent, rate $1$ Poisson processes,
\begin{equation}\label{eq:839}
\begin{aligned}
&\limsup_{n\to \infty} P\left(\sup_{0\le t \le T} |A^{n}_i(t)- \int_0^t a^n_i(s) ds| \ge \sqrt{n}T \rho^n_i \mbox{ for some }
i \in \{N,...,n\}\right)\\
&\le \limsup_{n\to \infty} P\left(\sup_{0\le t \le 2nT} |\cln_i(t)- t| \ge \sqrt{n}T \rho^n_i \mbox{ for some }
i \in \{N,...,n\}\right)\\
&\le \limsup_{n\to \infty} P(\rho^n_N <2) + \limsup_{n\to \infty} P\left(\sup_{0\le t \le n\rho^n_iT} |\cln_i(t)- t| \ge \sqrt{n}T \rho^n_i \mbox{ for some }
i \in \{N,...,n\}\right).
\end{aligned}
\end{equation}
Also, for $N\ge \tilde N$,
\begin{equation}\label{eq:837}
\limsup_{n\to \infty} P(\rho^n_N <2) \le \limsup_{n\to \infty} P(x_N \le M_1+ T(4+a)) \le \kappa.
\end{equation}
Next, let $\tilde \theta:= (A_1\exp\{A_2T(a+ M_1/T)/2\})^{-1}$. Then
\begin{align*}
&\limsup_{n\to \infty} P\left(\sup_{0\le t \le n\rho^n_iT} |\cln_i(t)- t| \ge \sqrt{n}T \rho^n_i \mbox{ for some }
i \in \{N,...,n\}\right)\\
&\le \limsup_{n\to \infty} P\left(\sup_{0\le t \le n\rho^n_iT} |\cln_i(t)- t| \ge \sqrt{n}T \rho^n_i \mbox{ for some }
i \in \{N,...,n\}; \sum_{i=N}^n e^{-A_2 x_i/2} \le \kappa \tilde \theta\right)\\
&\quad \quad + \limsup_{n\to \infty} P(\sum_{i=N}^n e^{-A_2 x_i/2} > \kappa \tilde \theta).
\end{align*}
Therefore, using \eqref{condnn} and \eqref{eq:poitail},
\begin{equation}\label{eq:834}
\begin{aligned}
&\limsup_{N\to \infty}\limsup_{n\to \infty} P\left(\sup_{0\le t \le n\rho^n_iT} |\cln_i(t)- t| \ge \sqrt{n}T \rho^n_i \mbox{ for some }
i \in \{N,...,n\}\right)\\
&\le  \limsup_{N\to \infty}\limsup_{n\to \infty} E\left(1\Big\{\sum_{i=N}^n e^{-A_2 x_i/2} \le \kappa \tilde \theta\Big\} (\tilde \theta)^{-1}\sum_{i=N}^n e^{-A_2 x_i/2}\right) \le \kappa.
\end{aligned}
\end{equation}
Combining \eqref{eq:837}, \eqref{eq:834}, and \eqref{eq:839}, we have, with the above choice of $N_1$,
\begin{equation}\label{eq:845}
\limsup_{N\to \infty}\limsup_{n\to \infty} P\left(\sup_{0\le t \le T} |A^{n}_i(t)- \int_0^t a^n_i(s) ds| \ge \sqrt{n}T \rho^n_i \mbox{ for some }
i \in \{N,...n\}\right) \le 2\kappa.
\end{equation}
A very similar calculation shows that, with $N_1$ as above,
\begin{equation}\label{eq:846}
\limsup_{N\to \infty}\limsup_{n\to \infty}P\left(\sup_{0\le t \le T} |\mathring{D}^{n}_i(t)- nt| \ge \sqrt{n}T \rho^n_i \mbox{ for some }
i \in \{N,...,n\}\right) \le 2\kappa.
\end{equation}
Combining \eqref{eq:845}, \eqref{eq:846}, \eqref{eq:823}, and \eqref{eq:822}
\begin{equation}\label{eq:851}
\limsup_{N\to \infty}\limsup_{n\to \infty}P\left(\inf_{i \in \{N,...,n\}}\inf_{0\le t \le T} \hat X^{n}_i(t) \le x_{N_1}\right)
\le 5\kappa.
\end{equation}
Finally, combining  \eqref{eq:taunn1n}, \eqref{eq:818}, and \eqref{eq:851}, with the above choice of $N_1$,
$$
\limsup_{N\to \infty} \limsup_{n\to \infty} P(\tau^{N,N_1,n} < T) \le 10\kappa K + 5 \kappa \le \kappa_0.
$$
This proves \eqref{eq:857} and the result follows.
\end{proof}
We now present a result that follows from a minor modification of \cite[Theorem 1]{banbudest}. The main difference in the proposition below from the result in
\cite{banbudest} is that the latter considers arrival rates to various queues of the form $n - a_i\sqrt{n}$ whereas the proposition below considers a setting where arrival rates have the form $n - a_i^n\sqrt{n}$ with $a_i^n \to a_i$ as $n\to \infty$.  We remark that \cite[Theorem 1]{banbudest} provides a convergence result for the gaps between ranked processes (along with the the process for the shortest queue) whereas the result below is formulated in terms of ranked processes themselves. One can go from one limit theorem to the other by an application of the continuous mapping theorem. We also note that the limit process in \cite[Theorem 1]{banbudest} is described as an $N$-dimensional reflected Brownian motion in the nonnegative orthant with certain oblique reflections at the boundary whereas, in the theorem below, the limit is described in terms of the ordering (ranking) of $N$ one dimensional reflected Brownian motions. The rewriting of the latter in the form as given in  \cite[Theorem 1]{banbudest} can be done through an application of Tanaka's formula (cf. \cite[Lemma 4]{pitmanpal} and \cite[Remark 1]{banbudest}). Proof modifications and additional steps needed to prove the proposition below using
\cite[Theorem 1]{banbudest} are standard and therefore we omit the proof of the proposition.
\begin{proposition}
\label{prop:350}
Let $\{X^n_i(0), i \in [n]\}$ and $\mu$ be as in Theorem \ref{convergencethm}. Fix $N\in \NN$ and let
$\mu^{(N)}$ be the marginal of $\mu$ on $\RR_+^N$, namely
$$\mu^{(N)}(A):= \mu\{ \bs{x}=(x_1, x_2, \ldots) \in \RR_+^{\infty}: (x_1, \ldots, x_N) \in A\}.$$
Let $\bs{X}^{(N)}$ be the weak solution of \eqref{eq:unranksde} (with $i \in [N]$), with initial distribution $\mu^{(N)}$, as introduced in Remark \ref{rem:weakn}. Let $\bs{X}^{(N)}_{\uparrow} := (X^{(N)}_{(1)}, \ldots , X^{(N)}_{(N)})$.
Then,  as $n \to \infty$,
$$(\hat X^{n,N}_{(1)}(\cdot), \ldots , \hat X^{n,N}_{(N)}(\cdot))$$
converges in distribution in $D([0,\infty), \RR_+^N)$ to
$\bs{X}^{(N)}_{\uparrow}$.
\end{proposition}
We now give a representation for the ranked Brownian systems $\bs{X}_{\uparrow}$ and $\bs{X}^{(N)}_{\uparrow}$ in Definition \ref{def: reflected infinite atlas model} and Remark \ref{rem:weakn} as singularly interacting diffusions that will be
used in the proof of Theorem \ref{convergencethm}.

Let $\mu \in \clp(\Lambda_0)$ and $(\Om, \clf, \PP, \{\clf_t\}_{t\ge 0})$ be a filtered probability space on which are given a collection $\{W_i, i \in \NN\}$
of Brownian motions and a $\clf_0$-measurable $\Lambda_0$-valued random variable
$\bs{X}^0 = (X_{1}^0, X_2^{0}, \ldots)\sim \mu$.

For $N\ge 1$, let $\tilde{\bs{X}}^{(N)} = (\tilde X_1^{(N)}, \ldots, \tilde X_N^{(N)})$ be given as the solution of the following system of equations.
\begin{equation}\label{eq:328}
\begin{aligned}
\tilde X_1^{(N)}(t) &= X_1^0  + b t + \sqrt{2}\,B_1(t) + L^{(N)}_1(t) - \tfrac12 L^{(N)}_2(t),\\[4pt]
\tilde X_i^{(N)}(t) &= X_i^0  + \sqrt{2}\,B_i(t) + \tfrac12 L_i^{(N)}(t) - \tfrac12 L_{i+1}^{(N)}(t),
\qquad 2 \le i \le N-1,\\[4pt]
\tilde X_N^{(N)}(t) &= X_N^0 + \sqrt{2}\,B_N(t) + \tfrac12 L_N^{(N)}(t).
\end{aligned}
\end{equation}
where $\{L_i^{(N)}\}$ are continuous, $\clf_t$-adapted, nondecreasing processes starting from $0$ such that
\begin{equation}\label{eq:loctim}
\int_0^{\infty} \tilde X_1^{(N)}(s)\, dL_1^{(N)}(s) = 0,
\qquad
\int_0^{\infty} \bigl(\tilde X_i^{(N)}(s)-\tilde X_{i-1}^{(N)}(s)\bigr)\, dL_i^{(N)}(s) = 0,
\quad 2 \le i \le N.
\end{equation}
From the work of Harrison and Reiman \cite{HarrisonReiman1981} it is known that the above system of equations has a unique pathwise solution. Furthermore, by an application of Tanaka's formula (cf. \cite[Remark 1]{banbudest}) it follows that
$\tilde{\bs{X}}^{(N)}$ has the same distribution as $\bs{X}^{(N)}_{\uparrow}$.

Next, consider a related infinite system of equations for continuous
$\RR_+^{\infty}$ valued processes $\tilde{\bs{X}} = (\tilde X_1, \tilde X_2, \ldots)$, as follows
\begin{equation}\label{eq:344}
\begin{aligned}
\tilde X_1(t) &= X_1^0  + b t + \sqrt{2}\,B_1(t) + L_1(t) - \tfrac12 L_2(t),\\[4pt]
\tilde X_i(t) &= X_i^0  + \sqrt{2}\,B_i(t) + \tfrac12 L_i(t) - \tfrac12 L_{i+1}(t),
\qquad i\ge 2,
\end{aligned}
\end{equation}
where $L_i$ are continuous, $\clf_t$-adapted, nondecreasing processes starting from $0$ such that \eqref{eq:loctim} is satisfied with $L_i^{(N)}$ replaced with $L_i$ and $2 \le i \le N$ replaced with $i\ge 2$. Also note that the first $N$ driving Brownian motions are the same in both systems.
It can then be shown that, for any $\mu \in \clp(\Lambda_0)$, there is a $\clf_t$-adapted, continuous, $\RR_+^{\infty}$ valued process  $\tilde{\bs{X}} = (\tilde X_1, \tilde X_2, \ldots)$ such that
for each $K \in \NN$,
$$(\tilde X^{(N)}_1, \ldots \tilde X^{(N)}_2, \ldots , \tilde X^{(N)}_K) \to (\tilde X_1, \tilde X_2, \ldots , \tilde X_K),$$
uniformly on compacts, a.s. and that the limit process $\tilde{\bs{X}}$ 
solves the system of equations in \eqref{eq:344}. Following Sarantsev \cite{Sarantsev2017InfiniteSystems}, we refer to this limit process as the {\em strong approximative solution of \eqref{eq:344}.} The proof of this fact follows exactly as that of \cite[Theorem 3.7]{Sarantsev2017InfiniteSystems}. It relies on the observation that for any $i \in \mathbb{N}$ and $t \in [0,T]$, $\{\tilde X^{(N)}_i(t) : N \ge 3\}$ forms a non-increasing sequence when constructed from the same sequence of Brownian motions.
This observation follows from (\cite{Sarantsev2019}, Corollary 3.9), which can be extended to the case of competing \textit{reflected} brownian motions when one notes that the proof still holds if one takes the diffusion, drift, and collision parameters to be zero for the first particle, which results in a stationary particle at zero that the others reflect off of. As this sequence is non-negative, we can construct
$$
\tilde X_i(t) := \lim_{N \to \infty}\tilde X^{(N)}_i(t), \quad i \in \mathbb{N}, t \in [0,T].
$$
Upon further observing that the local times also have a limit (using similar monotonicity properties), this produces the limiting family of processes $\{\tilde X_i : i \in \mathbb{N}\}$.

Furthermore,
if $\bs{X}$ is the unique weak solution of \eqref{eq:unranksde} with initial distribution $\mu$ then $\bs{X}_{\uparrow}$ has the same distribution on $C([0,\infty), \RR_+^{\infty})$ as the strong approximative solution of \eqref{eq:344} introduced above (this follows along the same lines as \cite[Lemma 3.5]{Sarantsev2017InfiniteSystems} upon incorporating reflection at the origin).
 Combining the above observations we immediately have the following result.

 \begin{proposition}
 \label{prop:434}
Let $\{X^n_i(0), i \in [n]\}$ and $\mu$ be as in Theorem \ref{convergencethm}. For $N\in \NN$, let $\bs{X}^{(N)}$ be as in Proposition \ref{prop:350}.
Let $\bs{X}$ be the unique weak solution of \eqref{eq:unranksde} with initial distribution $\mu$. Then, for each $K\in \NN$, as $N \to \infty$,
$$(X^{(N)}_{(1)}, \ldots , X^{(N)}_{(K)})$$
converges in distribution, in $C([0,\infty), \RR_+^K)$, to
$$(X_{(1)}, \ldots , X_{(K)}).$$
\end{proposition}
We can now complete the proof of Theorem \ref{convergencethm}. 

\noindent{\em Completing the proof of Theorem \ref{convergencethm}.}
It suffices to show that for each $K \in \NN$,
$(\hat X^n_{(1)}, \ldots \hat X^n_{(K)})$ converges in distribution, in $D([0,\infty), \RR_+^K)$, to $(X_{(1)}, \ldots , X_{(K)})$.
Consider a metric that metrizes the weak convergence topology on $D([0,\infty), \RR_+^K)$ (e.g. the bounded-Lipschitz distance) and denote it by $\bs{d}_K$.
Let $\mu^{n,K}$ (resp. $\mu^K$) be the probability law of $(\hat X^n_{(1)}, \ldots \hat X^n_{(K)})$ (resp. $(X_{(1)}, \ldots , X_{(K)})$) on $D([0,\infty), \RR_+^K)$.
Also, with $\hat X^{n,N}_{(i)}$ as in Lemma \ref{comparisonlemmaR},
denote the probability law of $(\hat{X}_{(1)}^{n,N}, \ldots , \hat{X}_{(K)}^{n,N})$
as $\mu^{n,N, K}$. Finally, with $X^{(N)}_{(i)}$ as in Proposition \ref{prop:350},
let
$\mu^{(N), K}$ be the probability law of $(X^{(N)}_{(1)}, \ldots , X^{(N)}_{(K)})$.
Then
\begin{align*}
\bs{d}_K(\mu^{n,K}, \mu^K) \le \bs{d}_K(\mu^{n,K}, \mu^{n,N, K}) + \bs{d}_K(\mu^{n,N, K}, \mu^{(N), K}) + \bs{d}_K(\mu^{(N), K},\mu^K).
\end{align*}
Consequently,
\begin{align*}
\limsup_{n\to \infty}
\bs{d}_K(\mu^{n,K}, \mu^K) &\le \limsup_{n\to \infty} \bs{d}_K(\mu^{n,K}, \mu^{n,N, K}) + \limsup_{n\to \infty} \bs{d}_K(\mu^{n,N, K}, \mu^{(N), K}) + \bs{d}_K(\mu^{(N), K},\mu^K)\\
& = \limsup_{n\to \infty} \bs{d}_K(\mu^{n,K}, \mu^{n,N, K}) + 
 \bs{d}_K(\mu^{(N), K},\mu^K).
\end{align*}
where the second line is from Proposition \ref{prop:350}. Taking $\limsup_{N\to \infty}$ in the above display,
\begin{align*}
\limsup_{n\to \infty}
\bs{d}_K(\mu^{n,K}, \mu^K) &\le \limsup_{N\to \infty}\limsup_{n\to \infty} \bs{d}_K(\mu^{n,K}, \mu^{n,N, K}) + 
\limsup_{N\to \infty}\bs{d}_K(\mu^{(N), K},\mu^K)\\
&=0,
\end{align*}
where the last line is from Lemma \ref{comparisonlemmaR} and Proposition \ref{prop:434}.
The result follows.
\hfill \qed

\subsection{Proof of Theorem \ref{convergencethm2}.}
\label{proofofdiffusionlimitsect2}
For fixed $N \in \NN$, define processes $Y_i^{n,N}$, $i\in [N]$ by equation \eqref{eq:yeqns}, setting $Y^{n,N}_{0}(t)=0$ and $Y^{n,N}_{N+1}(t)=\infty$
for all $t\ge 0$. As before, define diffusion scaled processes
$\hat Y^{n,N}_i(t) = \frac{Y^{n,N}_i(t)}{\sqrt{n}}$, $t \ge 0$, $i \in [N]$. Then we have the following analog of Proposition \ref{prop:350}.
\begin{proposition}
\label{prop:3502}
Let $\{X^n_i(0), i \in [n]\}$, $\mu$,
$\mu^{(N)}$, $\bs{X}^{(N)}$, $\bs{X}^{(N)}_{\uparrow}$ be as in Proposition \ref{prop:350}.
Then,  as $n \to \infty$,
$$\hat{\bs{Y}}^{n,N}:=(\hat Y^{n,N}_{1}, \ldots , \hat Y^{n,N}_{N})$$
converges in distribution in $D([0,\infty), \RR_+^N)$ to
$\bs{X}^{(N)}_{\uparrow}$.
\end{proposition}
\begin{proof}
Consider the $N\times N$ reflection matrix $R$ defined as
\[
R =
\begin{pmatrix}
1 & -\tfrac12 & 0 & \cdots & 0 \\
-1 & 1 & -\tfrac12 & \cdots & 0 \\
0 & -\tfrac12 & 1 & -\tfrac12 &  \vdots  \\
\vdots & \vdots & \ddots & \ddots  \\
0 & 0 & \cdots & -\tfrac12 & 1
\end{pmatrix}.
\]
Denote the Skorohod map on $\RR_+^N$ associated with the above reflection matrix as $\Gamma_R$, namely $\Gamma_R$ is a map from collection of functions $\psi$ in $D([0,\infty): \RR^N)$, satisfying $\psi(0) \in \RR_+^N$, to $D([0,\infty): \RR_+^N)$
such that $\Gamma_R(\psi) = \phi$ if and only if $\phi(t) = \psi(t) + R\eta(t)$, $t\ge 0$, where $\eta = (\eta_1, \ldots, \eta_N)$ is a $\RR_+^N$ valued function such that $\eta_i(0)=0$ and  $\eta_i$ is nondecreasing, and 
$\int_{[0,\infty)} \phi_i(t) d\eta_i(t) =0$ for every $i \in [N]$. From \cite{HarrisonReiman1981} it is known that the Skorohod map for the above reflection matrix is well defined and has the Lipschitz property: There is a $L_R \in (0, \infty)$ such that for any $T<\infty$, and $\psi, \tilde \psi$ as above
$$
\sup_{0\le t \le T} |\Gamma_R(\psi)(t) - \Gamma_R(\tilde \psi)(t)| \le L_R \sup_{0\le t \le T} |\psi(t) - \tilde \psi(t)|.
$$
Note that
\begin{equation*}
\hat Y_i^{n,N}(t) = \hat x^n_i + M_i^{n,N,A}(t) - M_i^{n,N,D}(t) + n^{-1/2} (\gamma^n_i -n)t - n^{-1}\gamma^n_i L^{n,N}_{i+1}(t) +  L^{n,N}_i(t),
\end{equation*}

where, with $\tilde\cla_i^c(ds\, dr) = \tilde\cla_i(ds\, dr) - ds\, dr$ and $\tilde\cld_i^c(ds\, dr) = \tilde\cld_i(ds\, dr) - ds\, dr$,
\begin{align*}
M_i^{n,N,A} &= n^{-1/2}\int_{[0,t]\times \RR_+} 1_{[0, \gamma^n_i]}(r)1_{\{Y^n_{i}(s-)<Y^n_{i+1}(s-)\}} \tilde \cla_i^c(ds\, dr),\\
M_i^{n,N,D} &= n^{-1/2}\int_{[0,t]\times \RR_+} 1_{[0, n]}(r)1_{\{Y^n_{i}(s-)>Y^n_{i-1}(s-)\}} \tilde \cld^c_i(ds\, dr)
\end{align*}
and
$$L^{n,N}_{i}(t) = n^{1/2}(1+ 1_{\{i\in \{2,...,N\}\}})\int_0^t 1_{\{Y^{n,N}_{i}(s)=Y^n_{i-1}(s)\}} ds, \; i \in [N]; \;\; L^{n,N}_{N+1}(t) =0.$$

Rearranging terms, and letting $M_i^{n,N,A} - M_i^{n,N,D}= M_i^{n,N}$ and $\beta^{n,N}_i= b1_{\{i=1\}}- an^{-1}$.
\begin{equation*}
\hat Y_i^{n,N}(t) = \hat x^n_i + M_i^{n,N}  + \beta^{n,N}_it - \frac{1}{2}n^{-1/2}\beta^{n,N}_i L^{n,N}_{i+1}(t) -  \frac{1}{2}L^{n,N}_{i+1}(t) +  \frac{1}{2}(1+1_{\{i=1\}})L^{n,N}_i(t),
\end{equation*}

Now define
\begin{align*}
\hat V_1^{n,N}(t) &= \hat x^n_1 + M_1^{n,N}(t)  + \beta^{n,N}_1 t  - \delta^{n,N}_1(t) \\
\hat V_i^{n,N}(t) &= \hat x^n_i- \hat x^n_{i-1} + M_i^{n,N}(t) - M_{i-1}^{n,N}(t) 
+ (\beta^{n,N}_i-\beta^{n,N}_{i-1})t- \delta^{n,N}_i(t), \; i \in \{2,... N\} 
\end{align*}
where, with $\beta^{n,N}_{0} =0$,
\begin{align*}
\delta^{n,N}_i(t) =  \frac{1}{2}n^{-1/2}\beta^{n,N}_{i} L^{n,N}_{i+1}(t) - \frac{1}{2}n^{-1/2}\beta^{n,N}_{i-1} L^{n,N}_{i}(t).
\end{align*}

Then, it follows that, with $\hat{\bs{V}}^{n,N} = (\hat V_1^{n,N}, \ldots , \hat V_N^{n,N})$,
\begin{equation}\label{eq:1021}
\hat{\bs{U}}^{n,N} = \Gamma_R(\hat{\bs{V}}^{n,N}).
\end{equation}
By a standard application of Aldous-Kurtz criterion we see that
$\hat{\bs{V}}^{n,N}$ is $C$-tight in $D([0,\infty):\RR^N)$. From the Lipschitz property of the Skorohod map it follows that $(\hat{\bs{U}}^{n,N}, \bs{L}^{n,N})$ is also $C$-tight in
$D([0,\infty):\RR_+^N\times \RR_+^N)$, where $\bs{L}^{n,N}= (L_1^{n,N}, \ldots, L_N^{n,N})$.
This in particular shows that, for each $i \in [N]$,
$M_i^{n,N} - \mathring{M}_i^{n,N}$ and $\delta^{n,N}_i$ converges to $0$ in probability in $D([0,\infty):\RR)$, where
$$\mathring{M}_i^{n,N}= \mathring{M}_i^{n,N,A}- \mathring{M}_i^{n,N,D},$$
and
\begin{align*}
\mathring{M}_i^{n,N,A} = n^{-1/2}\int_{[0,t]\times \RR_+} 1_{[0, \gamma^n_i]}(r) \tilde \cla_i^c(ds\, dr), \; \mathring{M}_i^{n,N,D} = n^{-1/2}\int_{[0,t]\times \RR_+} 1_{[0, n]}(r) \tilde \cld^c_i(ds\, dr), \; i \in [N].
\end{align*}
Consequently, as $n \to \infty$,
$$(\hat{\bs{X}}^{n,N}(0),  {\bs{M}}^{n,N}, \bs{\delta}^{n,N}) \Rightarrow (\bs{X}^{0,N}, \sqrt{2} \bs{B}^N, 0),
$$
where ${\bs{M}}^{n,N} = (M_1^{n,N}, \ldots, M_N^{n,N})$,
$\bs{\delta}^{n,N} = (\delta_1^{n,N}, \ldots, \delta_N^{n,N})$,
$ \bs{X}^{0,N} = (X^0_1, \ldots, X^0_N)$, and $\bs{B}^N = (B_1, \ldots, B_N)$.
Also, with $\bs{\beta}^N = (b, 0, \ldots)$, and  $\bs{\beta}^{n,N} = (\beta^{n,N}_1, \ldots, \beta^{n,N}_N)$, $\bs{\beta}^{n,N} \to \bs{\beta}^N$ as $n\to \infty$.

Recall from Section \ref{sec:origdyn} that $\bs{X}^{(N)}_{\uparrow}$  has the same distribution as $\tilde{\bs{X}}^{(N)}$ and consequently $\bs{Z}^{(N)} = (Z^{(N)}_1,\ldots,Z^{(N)}_N)$, where $Z^{(N)}_i := X_{(i)}^{(N)}- X_{(i-1)}^{(N)}$, $i \in [N]$, has the same distribution as $\tilde{\bs{Z}}^{(N)}= (\tilde Z_1^{(N)}, \ldots , \tilde Z_N^{(N)})$, where 
$\tilde Z_i^{(N)} = \tilde X_i^{(N)}- \tilde X_{i-1}^{(N)}$, $i \in [N]$, and we take
$X_{(0)}^{(N)}=\tilde X_{0}^{(N)}=0$. Furthermore, from the evolution equations for $\tilde{\bs{X}}^{(N)}$ in \eqref{eq:328} we see that
$$
\tilde{\bs{Z}}^{(N)} = \Gamma_R( {\bs{V}}^{N})$$
where ${\bs{V}}^{N} = (V_1^N, \ldots V_N^N)$ and 
\begin{align*}
V_1^{N}(t) &= X^0_1 + \sqrt{2} B_1(t)  + \beta^{N}_1 t   \\
V_i^{N}(t) &= X^0_{i}- X^0_{i-1} + \sqrt{2} B_i(t) - \sqrt{2} B_{i-1}(t)
+ (\beta^{N}_i-\beta^{N}_{i-1})t, \; i \in \{2,..., N\}.
\end{align*}
Using the Lipschitz property of the Skorohod map and \eqref{eq:1021}, we now have that
$\hat{\bs{U}}^{n,N}\Rightarrow \bs{Z}^{(N)}$. The convergence stated in the proposition now follows from the continuous mapping theorem.


\end{proof}
We now give an analog of Lemma \ref{comparisonlemmaR}. Since the proof follows along similar lines, we suppress some arguments.
We begin by introducing an {\em unordered system} such that the associated ranked system has the same distribution as $\bs{Y}^n = (Y^n_1, \ldots , Y^n_n)$.
The state processes for the unordered system will be denoted as $\clx^n_i$, $i \in [n]$, that describe a collection of $n$-queues for which, at time instants $t$ at which the $i$-th queue
is not tied with any other queues, the arrival rate to  it  is
$n -an^{-1/2} + b\sqrt{n} 1_{\{\clx^n_i(t) = \min_{j\in [n]}\clx^n_j(t)\}}$
and the departure rate is $n$. However, at time instants when this queue is tied with some other queues, both the arrival rate and departure rate slow down by a factor of 
$$T^n_i(t) = \left(\sum_{j \in [n]} 1_{\{\clx^n_i(t) = \clx^n_j(t)\}}\right)^{-1}.$$
A precise evolution equation can be  given using the PRM that were used in \eqref{eq:xnit}:
\begin{equation}
\label{eq:xnitmo}
\clx^n_i(t) = x^n_i + A^{n,*}_i(t)  - D^{n,*}_i(t), \; i \in [n], t\ge 0,
\end{equation}
where, with
$$a^{n,*}_i(t) = (n - an^{-1/2} + bn^{1/2}1_{\{\clx^{n}_i(t) = \min_{j \in [n]} \clx^{n}_j(t)\}})T^n_i(t);\;\; d^{n,*}_i(t) = n T^n_i(t).$$
\begin{align*}
A^{n,*}_i(t) &= \int_{[0,t]\times \RR_+} 1_{[0, a^{n,*}_i(s-)]}(r) \cla_i(ds\, dr), \; i \in [n]\\
D^{n,*}_i(t) &= \int_{[0,t]\times \RR_+} 1_{[0, d^{n,*}_i(s-)]}(r) 1_{\{\clx^n_i(s-) \neq 0\}} \cld_i(ds\, dr), \; i \in [n].
\end{align*}
It can then be seen that if $\bs{\clx}^n_{\uparrow} = (\clx^n_{(1)}, \ldots, \clx^n_{(n)})$
is the corresponding ranked system (with ties broken in lexicographical order),
$\bs{\clx}^n_{\uparrow}$ has the same distribution as $\bs{Y}^n$. The state processes of the analogous systems with a fixed $N$ will be denoted by $\bs{\clx}^{n,N}$ and
$\bs{\clx}^{n,N}_{\uparrow}$ and one has that
$\bs{\clx}^{n,N}_{\uparrow}$ has the same distribution as $\bs{Y}^{n,N}$.
As before, diffusion scaled processes will be denoted with a {\em hat}.
\begin{lemma}
\label{comparisonlemmaR2}
    Let  $\{X^n_i(0), i \in [n]\}$ and  $\mu$ be as in the statement of Theorem \ref{convergencethm}. Fix  $T>0$ and $K \in \NN$. 
Then there exists a sequence $\{\epsilon_{N}, N \in \NN\}$ such that $\epsilon_{N}\rightarrow 0$ as $N\to \infty$, and
\begin{equation}
\label{comparisoneqnR2}
    \liminf_{n \rightarrow \infty} P\Big((\hat{\clx}_{(1)}^{n}(t),...,\hat{\clx}^{n}_{(K)}(t))=(\hat{\clx}_{(1)}^{n,N}(t),...,\hat{\clx}_{(K)}^{n,N}(t)) \hspace{5mm} \forall t \in [0,T]\Big) \geq 1-\epsilon_{N}.
\end{equation}
\end{lemma}
\begin{proof}
Define the event $A$ and stopping times $\tau^{N, N_1, n}$ as in the proof of 
Lemma \ref{comparisonlemmaR}, replacing $X$ everywhere with $\clx$. It suffices to show the statement in \eqref{eq:857}. We estimate the probability in this display as in \eqref{eq:taunn1n} (with $X$ replaced by $\clx$).
The first term in \eqref{eq:taunn1n} is estimated as follows. Let
$$\mathring{D}^{n,*}_i(t) = \int_{[0,t]\times \RR_+} 1_{[0, d^{n,*}_i(s-)]}(r) 1_{\{\clx^n_i(s-) \neq 0\}} \cld_i(ds\, dr).$$
Then
\begin{equation}\label{eq:811}
\begin{aligned}
P\left(\sup_{0\le t \le T} \hat{\clx}^{n}_i(t) \ge x_{N_1}\right) 
&\le P\left(\sup_{0\le t \le T}  |A^{n,*}_i(t) - \int_0^t a^{n,*}_i(s) ds| \ge \sqrt{n}T\alpha^{N_1,n}_i\right)\\
&\quad + P\left(\sup_{0\le t \le T}  |\mathring{D}^{n,*}_i(t) - \int_0^t d^{n,*}_i(s) ds| \ge \sqrt{n}T\beta^{N_1,n}_i\right),
\end{aligned}
\end{equation}
Since $T^n_i(t) \le 1$, for large $n$,
$\sup_{0\le t \le T}  |A^{n,*}_i(t) - \int_0^t a^{n,*}_i(s) ds|$ and 
$\sup_{0\le t \le T}  |\mathring{D}^{n,*}_i(t) - \int_0^t d^{n,*}_i(s) ds|$ are stochastically dominated by
$\sup_{0\le t \le 2nT} |\cln(t)-t|$ and $\sup_{0\le t \le nT} |\cln(t)-t|$, respectively.
Using this, exactly as in the proof of Lemma \ref{comparisonlemmaR} we have that the estimate in \eqref{eq:818} holds with $X$ replaced with $\clx$.

We now consider the second term in the estimate in \eqref{eq:taunn1n}.  Using again the fact that  $T^n_i(t) \le 1$, we proceed exactly as in \eqref{eq:822}--\eqref{eq:839} with $X^n, A^n, D^n$ replaced with $\clx^n, D^{n,*}, A^{n,*}$ respectively and obtain the estimate in \eqref{eq:845} with $A^n, a^n$ replaced with $A^{n,*}, a^{n,*}$, respectively.
The proof of the estimate in \eqref{eq:846}, with $\mathring{D}^{n}_i(t)- nt$ replaced with
$\mathring{D}^{n,*}_i(t)- \int_0^t d^{n,*}_i(s) ds$, is completed in a similar fashion.

The result follows.

\end{proof}

Finally, the proof of Theorem \ref{convergencethm2} is completed exactly as the proof of Theorem \ref{convergencethm} by using Lemma \ref{comparisonlemmaR2},  Proposition \ref{prop:3502} (instead of Lemma \ref{comparisonlemmaR} and Proposition \ref{prop:350}, respectively),  and Proposition \ref{prop:434}.

\section{Convergence of the Stationary Distributions}
\label{convergenceofstationariessect}
In this Section we prove Theorems \ref{thm:jack}, \ref{jackcor} and Corollaries \ref{cor:mjsq}, \ref{cor: mix}.

\subsection{Proof of Theorem \ref{thm:jack}.}
Let $\theta_i$ denote the total arrival rate to the $i$-th queue. 
Then $\theta_i$ solve (see, e.g. \cite{Kelly1979}) the traffic equations
\begin{equation}\label{eq:traffic}
\theta_i = \bar\lambda_i + \sum_{j=1}^k \theta_j P_{j,i},
\qquad 1 \le i \le k .
\end{equation}
We now argue that the solution of the traffic equations is given as
\begin{equation}
\label{claimjack}
\begin{aligned}
\theta _1 = \frac{\lambda_1\cdots\lambda_k}{\mu_2\cdots\mu_k},\;
\theta_2 = \frac{\lambda_1+\mu_2}{\lambda_1}\,\theta_1,\;
\theta_i =
\frac{\mu_2\cdots\mu_{i-1}}{\lambda_1\cdots\lambda_{i-1}}
(\lambda_{i-1}+\mu_i)\,\theta_1, \; 3 \le i \le k.
\end{aligned}
\end{equation}
To see this, consider first the traffic equation \eqref{eq:traffic} with $i=1$:
\[
\theta_1 = \theta_2 P_{2,1}
= \theta_2 \frac{\lambda_1}{\lambda_1+\mu_2},
\]
which yields
$
\theta_2 = (\lambda_1+\mu_2)(\lambda_1)^{-1}\,\theta_1$.
Next, \eqref{eq:traffic} with $i=2$ gives
\[
\theta_2 = \theta_1 P_{1,2} + \theta_3 P_{3,2}
= \theta_1 + \theta_3 \frac{\lambda_2}{\lambda_2+\mu_3}.
\]
Substituting the expression for $\theta_2$ and solving gives
$
\theta_3
= \mu_2(\lambda_1\lambda_2)^{-1}(\lambda_2+\mu_3)\theta_1$.
We now proceed recursively. Suppose $\theta_\ell$ has the form for $3\le \ell \le i < k$ given
in \eqref{claimjack}.
Using the traffic equation \eqref{eq:traffic} for $i$,
\[
\theta_i \;=\; \theta_{i-1}P_{i-1,i} \;+\; \theta_{i+1}P_{i+1,i}
\;=\;  \theta_{i-1}\frac{\mu_{i-1}}{\lambda_{i-2}+\mu_{i-1}} \;+\; \theta_{i+1}\frac{\lambda_i}{\lambda_i+\mu_{i+1}} .
\]
Hence
\[
\theta_{i+1}\frac{\lambda_i}{\lambda_i+\mu_{i+1}}
\;=\;
\theta_i-  \theta_{i-1}\frac{\mu_{i-1}}{\lambda_{i-2}+\mu_{i-1}}.
\]
Using the form of $\theta_\ell$  given
in \eqref{claimjack} for $\ell = i-1, i$, and using the convention  $\mu_2\cdots\mu_{i-2}=1$ if $i=3$,
\begin{align*}
\theta_{i+1}\frac{\lambda_i}{\lambda_i+\mu_{i+1}}
&=
\frac{\mu_2\cdots\mu_{i-1}}{\lambda_1\cdots\lambda_{i-1}}
(\lambda_{i-1}+\mu_i)\theta_1
\;-\;
\frac{\mu_2\cdots\mu_{i-2}}{\lambda_1\cdots\lambda_{i-2}}
(\lambda_{i-2}+\mu_{i-1})\theta_1 \,
\frac{\mu_{i-1}}{\lambda_{i-2}+\mu_{i-1}}
\end{align*}
Therefore,
\begin{align*}
\theta_{i+1}\frac{\lambda_i}{\lambda_i+\mu_{i+1}}
&=
\frac{\mu_2\cdots\mu_{i-1}}{\lambda_1\cdots\lambda_{i-1}}
(\lambda_{i-1}+\mu_i)\theta_1
\;-\;
\frac{\mu_2\cdots\mu_{i-1}}{\lambda_1\cdots\lambda_{i-2}}
\theta_1
\\
&=
\frac{\mu_2\cdots\mu_{i-1}}{\lambda_1\cdots\lambda_{i-1}}\theta_1
\Big[(\lambda_{i-1}+\mu_i)-\lambda_{i-1}\Big]
=
\frac{\mu_2\cdots\mu_i}{\lambda_1\cdots\lambda_{i-1}}\theta_1 .
\end{align*}
Solving yields
\[
\theta_{i+1}
=
\frac{\mu_2\cdots\mu_i}{\lambda_1\cdots\lambda_{i}}
(\lambda_i+\mu_{i+1})\theta_1.
\]
Thus we have shown the claim in \eqref{claimjack} for $i=2, \ldots k$.
Finally consider $\theta_1$. 
Substituting the explicit form of $\theta_k,\theta_{k-1}$ for $k\ge 3$ from \eqref{claimjack} into the traffic equation \eqref{eq:traffic} for $i=k$, gives
\[
\frac{\mu_2\cdots\mu_{k-1}}{\lambda_1\cdots\lambda_{k-1}}
(\lambda_{k-1}+\mu_k)\theta_1
=
\lambda_k
+
\frac{\mu_2\cdots\mu_{k-2}}{\lambda_1\cdots\lambda_{k-2}}
(\lambda_{k-2}+\mu_{k-1})\theta_1
\frac{\mu_{k-1}}{\lambda_{k-2}+\mu_{k-1}}.
\]
Simplifying gives
$
\theta_1 = (\lambda_1\cdots\lambda_k)(\mu_2\cdots\mu_k)^{-1}$.
This completes the proof of \eqref{claimjack}.

Next, define the traffic
intensity at the $i$-th queue by
\begin{equation}\label{eq:rho-def}
\rho_i \;:=\; \frac{\theta_i}{\bar\mu_i},
\end{equation}
A classical result for open Jackson networks \cite{Kelly1979} states that if
$
\rho_i < 1$ for all $1\le i\le k$,
then $\bs{Q}$ is positive recurrent and admits a unique stationary distribution,
which is of product form. In particular, the stationary distribution of $\bs{Q}$ on $\mathbb{Z}_+^k$ is given as
\[
\pi(q_1,\dots,q_k)
=
\prod_{i=1}^k (1-\rho_i)\,\rho_i^{\,q_i},
\qquad (q_1,\dots,q_k)\in\mathbb{Z}_+^k.
\]
The proof of the theorem is now completed in observing that
\[
\rho_i =\frac{\theta_i}{\bar\mu_i} = \frac{\mu_2\cdots\mu_{i-1}}{\lambda_1\cdots\lambda_{i-1}}
\,\theta_1 =  \prod_{j=i}^k \frac{\lambda_j}{\mu_j}.
\]

\subsection{Proof of Corollary \ref{cor: mix}.}
It is easy to verify that the initial queuelengths of the scaled queue process satisfy \eqref{condnn}
under $\tilde \pi^{(n)}_{a,b}$, namely,
$$\limsup_{N\to \infty} \limsup_{n\to \infty} \tilde \pi^{(n)}_{a,b}\left(\sum_{i=N}^{n} e^{-c\omega_i} > \delta\right) = 0 \mbox{ for all } c >0, \delta>0.
$$
Also, $\tilde \pi^{(n)}_{a,b}$ converges to $\nu_{a,b}$. 
Thus, from Theorem \ref{convergencethm}, as $n\to \infty$,
$\hat{\boldsymbol{X}}^n_{\uparrow}$, with initial condition $\tilde \pi^{(n)}_{a,b}$, converges in distribution, in $D([0,\infty), \RR_+^{\infty})$, to the stationary process ${\boldsymbol{X}}_{\uparrow}$, where  $\boldsymbol{X}$  is the weak solution of  \eqref{eq:unranksde} with initial distribution $\nu_{a,b}$ and with $\delta=b$. This, together with the continuous mapping theorem, proves the second part of the corollary. For the first part, note that from this weak convergence we also have that with $f, k, T$ as in the statement of the corollary,
\begin{align*}
&\limsup_{n\to \infty} \sup_{s,t \in [0,T]} \left| \EE f(\hat X^n_{(1)}(t), \ldots , \hat X^n_{(k)}(t))
- \EE f(\hat X^n_{(1)}(s), \ldots , \hat X^n_{(k)}(s))\right|\\
&= \sup_{s,t \in [0,T]} \left| \EE f(X_{(1)}(t), \ldots , X_{(k)}(t))
- \EE f(X_{(1)}(s), \ldots , X_{(k)}(s))\right| =0,
\end{align*}
where the last equality uses the stationarity of ${\boldsymbol{X}}_{\uparrow}$ under $\nu_{a,b}$.

\subsection{Proof of Corollary \ref{cor:mjsq}.}

We apply Theorem \ref{thm:jack} with $\lambda_i = \lambda_i^{(n)} = n - a_i\sqrt{n}$ and $\mu_i= \mu_i^{(n)} =n$, for $i \in [k]$. Denoting the traffic intensity for station $i$ as $\rho^{(n)}_i$,
\[
\rho_i^{(n)} = \prod_{j=i}^k \left(1-\frac{a_j}{\sqrt{n}}\right).
\]
Then, from Theorem \ref{thm:jack}, denoting the stationary distribution of $\bs{Q}$ as $\hat \pi^n$,
for $z_i \in \NN_0/\sqrt{n}$, $i \in [k]$,
\begin{align*}
\pi^{(n)}\!\left([z_1, \infty) \times \dots \times [z_k, \infty)\right)
&= \hat \pi^{(n)}\!\left([\sqrt{n}z_1, \infty) \times \dots \times  [\sqrt{n}z_k, \infty)\right)
=
\prod_{i=1}^k
\left(\rho_i^{(n)}\right)^{\sqrt{n}z_i}.
\end{align*}
Taking logarithms  yields
\[
\lim_{n\to \infty} \sum_{i=1}^k \sqrt{n}z_i
\sum_{j=i}^k \log\!\left(1-\frac{a_j}{\sqrt{n}}\right)
=
-\sum_{i=1}^k z_i \left(\sum_{j=i}^k a_j\right),
\]
which completes the  proof.
\subsection{Proof of Theorem \ref{jackcor}.}
We apply Theorem \ref{thm:jack} with
$k=n$, $\lambda_i= \gamma^n_i$ and $\mu_i=n$, $i\in [n]$. Denoting the traffic intensity at the $i$-th station as $\lambda^{(n)}_i$, we have
\[
\rho_i^{(n)}
=
\prod_{j=i}^n
\left(1-\frac{a}{n^{3/2}}+\frac{b}{\sqrt{n}}\mathbf{1}_{\{j=1\}}\right).
\]
The first part of the theorem is now immediate from Theorem \ref{thm:jack}. Now consider the second part.

For fixed $i$,
\begin{equation}\label{eq:1232}
\log \pi^{(n)}\!\left([\sqrt{n}z_1, \infty)\times \dots \times [\sqrt{n}z_i, \infty)\right)
=
\sum_{\ell=1}^i \sqrt{n}z_\ell
\sum_{j=\ell}^n
\log\!\left(1-\frac{a}{n^{3/2}}+\frac{b}{\sqrt{n}}\mathbf{1}_{\{j=1\}}\right).
\end{equation}
We use the fact that
$$\log \left (1 - \frac{a}{n^{3/2}}\right) = -\frac{a}{n^{3/2}} + \clr^n_1, \;\;
\log \left (1 - \frac{a}{n^{3/2}} + \frac{b}{\sqrt{n}}\right) = \frac{b}{\sqrt{n}}- \frac{a}{n^{3/2}}
+ \clr^n_2,
$$
where for some $c<\infty$,
$$|\clr^n_1| \le \frac{c}{n^3}, \;\; |\clr^n_2| \le \frac{c}{n}.$$
Using these estimates in \eqref{eq:1232}
$$
\log \pi^{(n)}\!\left([\sqrt{n}z_1, \infty)\times \dots \times [\sqrt{n}z_i, \infty)\right) 
= -\sum_{\ell=1}^i\frac{a(n-\ell+1)}{n}z_{\ell} + z_1 b + \tilde \clr^n_1 + \sqrt{n} z_1 \clr^n_2,
$$
where $|\tilde \clr^n_1| \le \frac{c}{n^{3/2}} \sum_{\ell=1}^i z_{\ell}$.
Thus, sending $n \to \infty$,
$$\lim_{n\to \infty}\log \pi^{(n)}\!\left([\sqrt{n}z_1, \infty)\times \dots \times [\sqrt{n}z_i, \infty)\right) = -a\sum_{\ell=1}^iz_{\ell} + z_1 b.
$$
The result follows.

\subsection{Proof of Corollary \ref{corstat}.}
We will suppress time in the associated stochastic processes since all the calculations are for time $0$.
Recall that, under the law $\widetilde\pi^{(n)}_{(a,b)}$ for the ordered queue-length process, the gap variables
\[
Z_i^{n} = X_{(i)}^{n} - X_{(i-1)}^{n}, \qquad X_{(0)}^{n} = 0,
\]
are independent, and their diffusion-scaled versions $\hat  Z_i^{n} \doteq Z_i^{n}/\sqrt n$, with $\pi^{(n)}$ the probability law of
$(\hat  Z_1^{n}, \ldots , \hat  Z_n^{n})$, satisfy
\[
\pi^{(n)}([z_1,\infty)\times\cdots\times[z_n,\infty))
= \prod_{i=1}^n \bigl(\rho_i^{(n)}\bigr)^{\sqrt n z_i},
\qquad z_i\in \mathbb{N}_0/\sqrt n,
\]
where
\[
\rho_i^{(n)}
= \prod_{j=i}^n \left(1-\frac{a}{n^{3/2}}+\frac{b}{\sqrt n}\mathbf{1}_{\{j=1\}}\right).
\]
Equivalently, each $Z_i^{n}$ is geometric on $\mathbb{N}_0$ with parameter $\rho_i^{(n)}$, and thus
$
\EE[Z_i^{n}] = \rho_i^{(n)}(1-\rho_i^{(n)})^{-1}$.
We now study the asymptotics of $\rho_i^{(n)}$. In what follows, for $\alpha \in \RR$, a term of the form
$O(n^{\alpha})$ is interpreted as a real sequence $c(n)$, such that $\sup_{n \in \NN}|c(n)|/n^{\alpha}< \infty$.

For $i=1$, we have
\[
\rho_1^{(n)}
=
\left(1-\frac{a}{n^{3/2}}+\frac{b}{\sqrt n}\right)
\left(1-\frac{a}{n^{3/2}}\right)^{n-1}.
\]
Using
$\log(1-x) = -x + O(x^2)$ for $x\to 0$,
we obtain
\[
(n-1)\log\left(1-\frac{a}{n^{3/2}}\right)
= -\frac{a(n-1)}{n^{3/2}} + O(n^{-2})
= -\frac{a}{\sqrt n} + O(n^{-3/2}),
\]
and therefore
\[
\left(1-\frac{a}{n^{3/2}}\right)^{n-1}
= 1-\frac{a}{\sqrt n} + O(n^{-1}).
\]
Hence
\[
\rho_1^{(n)}
=
\left(1+\frac{b}{\sqrt n}-\frac{a}{n^{3/2}}\right)
\left(1-\frac{a}{\sqrt n}+O(n^{-1})\right)
= 1-\frac{a-b}{\sqrt n} + O(n^{-1}),
\]
and thus
\begin{equation}\label{eq:est326}
1-\rho_1^{(n)} = \frac{a-b}{\sqrt n}\left(1+ O(n^{-1/2})\right),
\;\;
\EE[Z_1^{n}] = \frac{\sqrt n}{a-b}\left(1+ O(n^{-1/2})\right).
\end{equation}

Next,
for $i\ge 2$, 
$
\rho_i^{(n)} = \left(1-\frac{a}{n^{3/2}}\right)^{n-i+1}.
$
Since,
\[
(n-i+1)\log\left(1-\frac{a}{n^{3/2}}\right)
= -\frac{a(n-i+1)}{n^{3/2}} + O(n^{-2}),
\]
 we obtain for $i \ge 2$,
\begin{equation*}
1-\rho_i^{(n)} = \frac{a(n-i+1)}{n^{3/2}}\left(1+ O(n^{-1/2})\right),
\;\;
\EE[Z_i^{n}] = \frac{n^{3/2}}{a(n-i+1)}\left(1+ O(n^{-1/2})\right).
\end{equation*}
Thus we have that, for each fixed $i\ge 2$,
\begin{equation}\label{eq:est328}
\frac{a}{\sqrt{n}}\EE[Z_i^{n}] \rightarrow 1, \ \text{ as } \ n \rightarrow \infty.
\end{equation}

Consider now the  $k$-th ranked queue.
For fixed $k\ge 1$,
$
\EE[X_{(k)}^{n}]
= \sum_{i=1}^k \EE[Z_i^{n}]$.
Thus, for fixed $k$, combining \eqref{eq:est326} and \eqref{eq:est328},
\[
\frac{\EE[X_{(k)}^{n}]}{\sqrt n\left(\frac{1}{a-b} + \frac{k-1}{a}\right)}  \rightarrow 1, \ \text{ as } \ n \rightarrow \infty.
\]
This proves (i).

Now consider (ii).
Since
$
X_{(n)}^{n} - X_{(1)}^{n} = \sum_{i=2}^n Z_i^{n}$,
we have
\begin{align*}
\EE[X_{(n)}^{n} - X_{(1)}^{n}]
=
\sum_{i=2}^n \frac{n^{3/2}}{a(n-i+1)}\left(1+ O(n^{-1/2})\right)
= \frac{n^{3/2}}{a}\sum_{m=1}^{n-1} \frac{1}{m}\left(1+ O(n^{-1/2})\right).
\end{align*}
Thus,
\[
\frac{\EE[X_{(n)}^{n} - X_{(1)}^{n}]}{n^{3/2}a^{-1} \log n} \rightarrow 1, \ \text{ as } \ n \rightarrow \infty,
\]
proving (ii).

Finally consider part (iii).
Note that
$
\frac{1}{n}\sum_{k=1}^n X_{(k)}^{n}
= \frac{1}{n}\sum_{i=1}^n (n-i+1)Z_i^{n}$.
Thus,
\begin{align*}
\EE\left[\frac{1}{n}\sum_{k=1}^n X_{(k)}^{n}\right]
 &= \frac{\sqrt n}{a-b}\left(1+ O(n^{-1/2})\right) + \frac{1}{n}\sum_{i=2}^n (n-i+1) \frac{n^{3/2}}{a(n-i+1)}\left(1+ O(n^{-1/2})\right) \\
 &= \left(\frac{\sqrt n}{a-b} + \frac{n-1}{n}\frac{n^{3/2}}{a}\right)\left(1+ O(n^{-1/2})\right)\\
 &= \left(\frac{n^{3/2}}{a} + \frac{b\sqrt n}{a(a-b)}\right)\left(1+ O(n^{-1/2})\right).
\end{align*}
This proves (iii), and the corollary.

\section{Proof of Theorem \ref{stationarydistreflectedatlasresult}.}
\label{sec:invmzrpf}
The proof will be a consequence of Theorems \ref{convergencethm2} and Theorem \ref{jackcor}.
Fix $b>0$ and $a \in (b,\infty)$. 
Let $\mu_{a,b}$ be as in the statement of the theorem.
Recall that
$$\Lambda := \{\boldsymbol{z} = (z_1, z_2, \ldots) \in \RR_+^{\infty}: \boldsymbol{x}(\boldsymbol{z}) := (z_1, z_1+ z_2, \ldots) \in \Lambda_0\}$$
and
$$\Lambda_0 = \left\{\boldsymbol{x} = (x_1, x_2, \ldots) \in \RR_+^{\infty}: x_1 \le x_2 \le \cdots \mbox{ and }
\sum_{i=1}^{\infty} \exp(-\alpha x_i^2) <\infty \mbox{ for all } \alpha>0\right\}. 
$$

We begin by observing that $\mu_{a,b}$ is supported on $\clp(\Lambda)$. This can be verified by an application of the strong law of large numbers and we leave the details to the reader.
Recall the measure $\nu_{a,b}$ introduced above Corollary \ref{cor: mix}.
 Since $\mu_{a,b}$ is supported on $\clp(\Lambda)$ we have that
$\nu_{a,b} \in \clp(\Lambda_0)$.

Next suppose that the diffusion scaled initial queue length vector in Theorem \ref{convergencethm}, namely, 
$\hat{\bs{X}}^n = (\hat X^n_1(0), \ldots \hat X^n_n(0))$, is such that, with
$$\hat{\bs{Z}}^n(0) = (\hat Z^n_1(0), \ldots \hat Z^n_n(0)) = (\hat X^n_1(0), \hat X^n_2(0)-\hat X^n_1(0), \ldots, \hat X^n_n(0)- \hat X^n_{n-1}(0)),$$
\begin{equation}\label{eq:1017}
P(\hat{\bs{Z}}^n(0) \in [z_1, \infty)\times \cdots \times [z_n, \infty))
= \pi^{(n)}([z_1, \infty)\times \cdots \times [z_n, \infty)), \quad (z_1,\dots,z_n) \in \RR^\infty_+,
\end{equation}
where $\pi^{(n)}$ is as in Theorem \ref{jackcor}. With these initial conditions, we construct the process $(\hat{\boldsymbol{Y}}^n, \hat{\boldsymbol{U}}^n)$ as in Theorem \ref{convergencethm2}.
 We claim that Condition \ref{condnn} holds, namely
\begin{equation}\label{eq:826n}
\limsup_{N\to \infty} \limsup_{n\to \infty} P\left(\sum_{i=N}^{n} e^{-c\hat X^n_i(0)} > \delta\right) = 0 \mbox{ for all } c >0, \delta>0.
\end{equation}
Assuming the claim, we now complete the proof of the theorem.
Fix $k \in \NN$ and $t>0$. Let $f: \RR_+^k \to \RR$ be a continuous and bounded function.
From Theorem \ref{jackcor}, $\pi^{(n)} \to \mu_{a,b}$ weakly.
Also, from the claim above, and since $\nu_{a,b} \in \clp(\Lambda_0)$, we have, from Theorem \ref{convergencethm2}, for each $s\ge 0$, as $n \to \infty$
\begin{equation}\label{eq:839s}
E f(\hat U^n_1(s), \ldots , \hat U^n_k(s)) \to E f(Z_1(s), \ldots Z_k(s)),
\end{equation}
where $\bs{Z} = (Z_1, \ldots )$ is defined as above \eqref{eq:hatxhatz} with $\bs{X}$ given as the weak solution of 
\eqref{eq:unranksde} with initial condition $\nu_{a,b}$.
Next, applying Theorem \ref{jackcor} again we have that $\pi^{(n)}$ is a stationary distribution for $\hat{\boldsymbol{U}}^n$.
Thus we have that
$$E f(\hat U^n_1(t), \ldots , \hat U^n_k(t)) = E f(\hat U^n_1(0), \ldots , \hat U^n_k(0)).$$
Applying \eqref{eq:839s} we now have that
$$ E f(Z_1(t), \ldots Z_k(t)) =  E f(Z_1(0), \ldots Z_k(0)).$$

Since $k \in \NN$ and $t\ge 0$ are arbitrary, the result follows.

Finally we prove the claim. Fix $c>0$ and $\delta>0$.
By Markov's inequality
\begin{equation}\label{eq:1016}
P\left(\sum_{i=N}^{n} e^{-c\hat X^n_i(0)} > \delta\right) \le
\frac{1}{\delta} \sum_{i=N}^n E e^{-c\hat X^n_i(0)} =
\frac{1}{\delta} \sum_{i=N}^n E e^{-c\sum_{j=1}^i \hat Z^n_j(0)}.
\end{equation}
Because of \eqref{eq:1017} and the definition of $\pi^{(n)}$, we see that 
$\{\hat Z^n_j(0), j \in [n]\}$ are independent random variables with 
$Z^n_j(0) = \sqrt{n} \hat Z^n_j(0)$ distributed as $Geom(1- \rho_j^{(n)})-1$.
Thus we have that
$$
E e^{-c\hat Z^n_j(0)} = \frac{1- \rho_j^{(n)}}{1- \rho_j^{(n)}e^{-c/\sqrt{n}}}.
$$
Note that, for $j \ge 2$,
$$\rho_j^{(n)} = \left( 1 - \frac{a}{n^{3/2}}\right)^{n-j+1}$$
and thus $\rho_j^{(n)}$ is nondecreasing in $j$.
Combining this with the observation that for $u \in (0,1)$, the function $x \mapsto (1-x)/(1-xu)$ is a decreasing function, we see that,
for all $j\ge 2$,
$$
E e^{-c\hat Z^n_j(0)} \le \frac{1- \rho_2^{(n)}}{1- \rho_2^{(n)}e^{-c/\sqrt{n}}}
:= m_n(c).
$$
Next note that
$$\rho_2^{(n)} = \left( 1 - \frac{a}{n^{3/2}}\right)^{n-1} = 1- \frac{a}{\sqrt{n}} + o(n^{-1/2})$$
and therefore
$m_n(c) \to a/(a+c)$ as $n\to \infty$.
Fix $q(c) \in (a/(a+c), 1)$ and let $n_0 \in \NN$ be such that $m_n(c) \le q(c)$ for all $n \ge n_0$.
Then, for all $n\ge n_0$,
$$E e^{-c\hat X^n_i(0)} = \prod_{j=1}^iE e^{-c \hat Z^n_j(0)}
\le \prod_{j=2}^iE e^{-c \hat Z^n_j(0)} \le q(c)^{i-1}.$$
Therefore
$$\sum_{i=N}^n E e^{-c\hat X^n_i(0)} \le \sum_{i=N}^n q(c)^{i-1} \le \frac{q(c)^{N-1}}{1-q(c)}.
$$
Combining this with \eqref{eq:1016} we see that
$$
\limsup_{N\to \infty} \limsup_{n\to \infty}
P\left(\sum_{i=N}^{n} e^{-c\hat X^n_i(0)} > \delta\right) 
\le \limsup_{N\to \infty} \frac{1}{\delta} \frac{q(c)^{N-1}}{1-q(c)} =0.
$$
\hfill \qed

\section*{Acknowledgements}
 Banerjee was partially supported by NSF-CAREER award DMS-2141621.  Budhiraja was partially supported by NSF DMS-2152577 and DMS-2134107. Banerjee, Budhiraja and Loeser were partially funded by NSF RTG grant DMS-2134107. Part of this work was carried out while Loeser was a postdoctoral fellow at the University of North Carolina at Chapel Hill.

\bibliographystyle{amsplain}
\bibliography{loadbalancingbibliography}

\noindent{\scriptsize {\textsc{\noindent S. Banerjee, A. Budhiraja and E. Loeser,\newline
Department of Statistics and Operations Research\newline
University of North Carolina\newline
Chapel Hill, NC 27599, USA\newline
emails: sayan@email.unc.edu, budhiraj@email.unc.edu, ehloeser@unc.edu
\vspace{\baselineskip} } }}
\end{document}